\documentclass[11pt,a4paper,reqno]{amsart}

\usepackage[T1]{fontenc}
\usepackage[utf8]{inputenc}
\usepackage[english]{babel}

\allowdisplaybreaks

\usepackage{lmodern}
\usepackage{microtype}
\usepackage{geometry}

\usepackage{mathtools}

\usepackage{esint}

\usepackage{amsthm}

\usepackage[unicode=true,pdfusetitle]{hyperref}
\let\C\undefined

\usepackage[abbrev,backrefs]{amsrefs}
\usepackage[nameinlink]{cleveref}

\numberwithin{equation}{section}

\theoremstyle{plain}
\newtheorem{proposition}{Proposition}[section]
\newtheorem{theorem}[proposition]{Theorem}

\theoremstyle{definition}
\newtheorem{definition}[proposition]{Definition}

\theoremstyle{remark}

\crefname{myexample}{example}{examples}

\usepackage{enumitem}

\newcommand{\Rset}{\mathbb{R}}

\newcommand{\Deriv}{\mathrm{D}}
\newcommand{\dext}{\mathrm{d}}
\newcommand{\Nset}{\mathbb{N}}
\newcommand{\Sset}{\mathbb{S}}

\newcommand{\defeq}{\coloneqq}

\newcommand{\dif}{\;\mathrm{d}}

\DeclarePairedDelimiter{\abs}{\lvert}{\rvert}
\DeclarePairedDelimiter{\norm}{\lVert}{\rVert}
\DeclarePairedDelimiter{\seminorm}{\lvert}{\rvert}
\DeclarePairedDelimiter{\brk}{(}{)}

\DeclarePairedDelimiterX\dualprod[2]{(}{)}{#1\cdot #2}
\DeclarePairedDelimiterX\intvo[2]{(}{)}{#1, #2}
\DeclarePairedDelimiterX\intvc[2]{[}{]}{#1, #2}
\DeclarePairedDelimiterX\intvl[2]{(}{]}{#1, #2}
\DeclarePairedDelimiterX\intvr[2]{[}{)}{#1, #2}
\providecommand{\st}{\,\vert\,}

\newcommand\stSymbol[1][]{%
\nonscript\;#1\vert
\allowbreak
\nonscript\;
\mathopen{}}
\DeclarePairedDelimiterX\set[1]\{\}{%
\renewcommand\st{\stSymbol[\delimsize]}
#1
}
\DeclareMathOperator{\Lin}{Lin}

\crefname{subsection}{§}{§§}
\newcommand{\familyname}[1]{\textsc{#1}}

\usepackage{constants}

\begin{document}

\title[Limiting Sobolev estimates and cancelling differential operators]{Limiting Sobolev estimates for vector fields and cancelling differential operators}
\author{Jean Van Schaftingen}
\address{Université catholique de Louvain (UCLouvain), Institute de Recherche en Mathématique et Physique (IRMP), Chemin du Cyclotron 2 bte L7.01.01, 1348 Louvain-la-Neuve, Belgium}
\email{Jean.VanSchaftingen@UCLouvain.be}

\thanks{These lecture notes were written on the occasion of the Spring School on Analysis 2023 ``Function Spaces and Applications XII'' held at Paseky nad Jizerou from May 28 to June 3, 2023}


\begin{abstract}
These notes present Sobolev-Gagliardo-Nirenberg endpoint estimates for classes of homogeneous vector differential operators. Away of the endpoint cases, the classical Calder\'on--Zygmund estimates show that the ellipticity is necessary and sufficient to control all the derivatives of the vector field. In the endpoint case, Ornstein showed that there is no nontrivial estimate on same-order derivatives. 
On the other hand endpoint Sobolev estimates were proved for the deformation operator (Korn-Sobolev inequality by M.J. Strauss) and for the Hodge complex (Bourgain and
Brezis). The class of operators for which such Sobolev estimates holds can be characterized by a cancelling condition. The estimates rely on a duality estimate for $L^1$ vector fields satisfying some conditions on the derivatives, combined with classical algebraic and harmonic analysis techniques. 
This characterization unifies classes of known inequalities and extends to the case of Hardy inequalities.
\end{abstract}

\maketitle

\tableofcontents 

\section{Sobolev inequalities for vector fields}

\subsection{Sobolev inequality}
The \emph{Sobolev inequality} is a fundamental tool in the study of Sobolev spaces of weakly differentiable functions. 
It states that for every 
\(n, k \in \Nset \setminus \set{0}\), \(\ell \in \set{0, \dotsc, k - 1}\) and \(p \in  [1, \frac{n}{k - \ell})\) there exists a constant \(\Cl{cst_Eeh9quaeQu2iyeeng1jaechu} \in \intvo{0}{\infty}\) such that for each function \(u \in C^\infty_c (\Rset^n, \Rset)\) the inequality
\begin{equation}
\label{eq_AhR2ooxi0teixahheePoa5hu}
  \brk[\Big]{\int_{\Rset^n} \abs{\Deriv^{\ell} u}^\frac{np}{n - (k - \ell)p}}^{1 - \frac{(k - \ell)p}{n}}
  \le 
  \Cr{cst_Eeh9quaeQu2iyeeng1jaechu}
  \int_{\Rset^n} \abs{\Deriv^{k} u}^p
\end{equation}
holds. 
The original statement of the estimate \eqref{eq_AhR2ooxi0teixahheePoa5hu} when \(p > 1\) goes back to Sergei Lvovich \familyname{Sobolev}, with a proof based on what is called nowadays the Hardy--Littlewood--Sobolev interpolation inequality \cite{Sobolev_1938} (for a modern presentation, see for example \cite{Stein_1970}*{ch.\thinspace V \S 2.2}).
The case \(p = 1\) remained open for some time --- it is mentioned for example in the first edition of Laurent \familyname{Schwartz}'s \emph{Théorie des distributions} \citelist{\cite{Schwartz_1950}\cite{Schwartz_1951}} --- before being settled twenty years later by Emilio \familyname{Gagliardo} \cite{Gagliardo_1958} and Louis \familyname{Nirenberg} \cite{Nirenberg_1959}. 
Interestingly, thanks to its simplicity, Gagliardo and Nirenberg's proof has been chosen by many authors to prove pedagogically the full scale of Sobolev inequalities \eqref{eq_AhR2ooxi0teixahheePoa5hu} through a shrewd application of the chain rule (see for example \citelist{\cite{Brezis_2011}\cite{Willem_2013}}).
Also, when \(p = 1\),  \(k = 1\) and \(\ell = 0\), the Sobolev inequality \eqref{eq_AhR2ooxi0teixahheePoa5hu} turns out to be equivalent with the classical isoperimetric inequality \citelist{\cite{Federer_Fleming_1960}\cite{Mazya_2011}}; when \(1< p < n\) and \(\ell = 0\), the optimal constant can be computed, with optimizers when \(p = 2\) being entire solutions of Yamabe's problem of prescribed scalar curvature in differential geometry  \citelist{\cite{Aubin_1976}\cite{Talenti_1976}}.

Whereas for scalar functions it is clear that a component of the derivative \(\Deriv u\) cannot be dropped in the integral on the right-hand side of the Sobolev inequality \eqref{eq_AhR2ooxi0teixahheePoa5hu}, one could hope that when \(u\) is taken instead to be a \emph{vector field}, one could replace the total derivative  \(\Deriv^k u\) in the right-hand side by some sparser differential operator \(A(\Deriv) u\).
Here and in the sequel, given finite-dimensional linear spaces \(V\) and \(E\) a linear mapping \(\smash {A \in \Lin (\Lin_{\mathrm{sym}}^k (\Rset^n, V), E)}\), we define the differential operator \(A (\Deriv)\) acting on \(u \in C^\infty (\Rset^n, V)\) at any point \(x \in \Rset^n\) by application of the linear map \(A\) to the total \(k\)--th order derivative \(\Deriv^k u(x) : \Rset^n \to \Lin_{\mathrm{sym}}^k (\Rset^n, V)\), seen as a function to symmetric \(k\)--linear maps from \((\Rset^n)^k\) to \(V\):
\begin{equation*}
A (\Deriv)u (x)
\defeq 
A [\Deriv^k u (x)] \in E\;;
\end{equation*}
in terms of partial derivatives, for every multiindex \(\smash{\alpha = (\alpha_1, \dotsc, \alpha_n) \in \Nset^n}\) satisfying \(\smash {\abs{\alpha} \defeq \alpha_1 + \dotsb + \alpha_n = k}\), there exists some linear map \(A_\alpha \in \Lin (V, E)\) for which for each \(u \in C^\infty (\Rset^n, V)\) and each \(x \in \Rset^n\), we have
\begin{equation}
  \label{eq_Io6oigh0kee2ocaev}
  A (\Deriv) u (x) = \sum_{\substack{\alpha \in \Nset^n\\ \abs{\alpha} = k}} A_\alpha [\partial^\alpha u (x)] = \sum_{\substack{\alpha \in \Nset^n\\ \abs{\alpha} = k}} \partial^\alpha (A_\alpha [u]) (x)\;,
\end{equation}
where \(\partial^\alpha \defeq \partial_1^{\alpha_1} \dotsb \partial_n^{\alpha_n}\) and the second equality follows from the fact that the partial derivative \(\partial^\alpha\) commutes \(A_\alpha\) since the latter does not depend on the variable \(x\).

Given such a differential operator \(A(\Deriv)\), the goal of the present notes is to determine when for every \(u \in C^\infty (\Rset^n, V)\) the vector Sobolev inequality 
\begin{equation}
  \label{eq_bah2acooqueBei4Ahph}
   \brk[\Big]{\int_{\Rset^n} \abs{\Deriv^{\ell} u}^\frac{np}{n - (k - \ell)p}}^{1 - \frac{(k - \ell)p}{n}}
  \le 
  \C
  \int_{\Rset^n} \abs{A (\Deriv)[u]}^p
\end{equation}
holds.
The material covered in this work come essentially from the original references \citelist{\cite{VanSchaftingen_2008}\cite{VanSchaftingen_2013}\cite{Bousquet_VanSchaftingen_2014}};
they are a somehow more informal counterpart part of the recent lectures notes \cite{VanSchaftingen_2023}; they are nicely complemented by the survey article \cite{VanSchaftingen_2014}.

\subsection{\texorpdfstring{\(L^2\)}{L²} estimates}
In the case \(p = 2\), one has by the Parseval identity
\begin{equation}
  \label{eq_mi2eijaefaazae5io1ef7ohM}
\int_{\Rset^n} \abs{A (\Deriv) u}^2 = 
\int_{\Rset^n} \abs{\brk{2 \pi i}^k A (\xi) [\mathcal{F} u(\xi)]}^2  \dif \xi\;,
\end{equation}
where the Fourier transform \(\mathcal{F} u : \Rset^n\to V + i V\) of a Schwartz test function \(u \in \mathcal{S} (\Rset^n, V)\) is defined for every \(\xi \in \Rset^n\) by the integral formula
\begin{equation*}
(\mathcal{F} u) (\xi) 
\defeq 
\int_{\Rset^n} e^{-2\pi i\,\xi\cdot x}  \, u (x) \dif x\;,
\end{equation*}
so that 
\begin{equation}
  \label{eq_du0zoovuociuBi9nu}
\mathcal{F} (A (\Deriv) u) (\xi)
= A \brk[\big]{(2 \pi i\,\xi)^{\otimes k}  \otimes \mathcal{F} u (\xi) }
= \brk{2 \pi i}^k A (\xi)[\mathcal{F} u (\xi)]\;,
\end{equation}
and where for every \(\xi \in \Rset^n\) and \(v \in V\)
\begin{equation}
  \label{eq_deeque4ungie3eijeMooCiof}
  A (\xi) [v] \defeq A (\xi^{\otimes k} \otimes v)
  = \sum_{\substack{\alpha \in \Nset^n\\ \abs{\alpha} = k}} \xi^\alpha A_\alpha
  \in \Lin (V, E)\;
\end{equation}
is the symbol of the differential operator \(A (\Deriv)\) and where we have written \(\xi^\alpha \defeq \xi_1^{\alpha_1} \dotsm \xi_n^{\alpha_n} \in \Rset\)
for \(\alpha =  (\alpha_1, \dotsc, \alpha_n) \in \Nset^n\) and \(\xi= (\xi_1, \dotsc, \xi_n) \in \Rset^n\).

If we assume that the operator \(A (\Deriv)\) is \emph{injectively elliptic}, that is, that for every \(\xi \in \Rset^n\setminus \set{0}\), we have \(\ker A (\xi) = \set{0}\), then since \(\Rset^n\) and \(V\) are finite-dimensional vector space, we get from Weierstrass’s theorem the existence of a constant \(\Cl{cst_johdoov8eegheeH2hur6Apee} \in \intvo{0}{\infty}\) such that for every \(\xi \in \Rset^n\) and \(v \in \Rset^n\),
\begin{equation}
\label{eq_thahnopaeyavueTup7ohL3ai}
 \abs{\xi}^k \abs{v} \le \Cr{cst_johdoov8eegheeH2hur6Apee}
 \abs{A (\xi)[v]}\;.
\end{equation}
Combining the Parseval identity \eqref{eq_mi2eijaefaazae5io1ef7ohM} and the estimate \eqref{eq_thahnopaeyavueTup7ohL3ai}, we get
\begin{equation}
\label{eq_xie4hohngee0ooQuisahtoov}
\begin{split}
 \int_{\Rset^n} \abs{\Deriv^{k} u}^2
  &=   \int_{\Rset^n} \abs{2 \pi i\,\xi}^{2k} \abs{u(\xi)}^2\dif \xi\\
  &\le \Cr{cst_johdoov8eegheeH2hur6Apee}^2
  \int_{\Rset^n} \abs{\brk{2 \pi i}^k A (\xi) [u(\xi)]}^2\dif \xi
  =\Cr{cst_johdoov8eegheeH2hur6Apee}^2
  \int_{\Rset^n}  \abs{A (\Deriv) u}^2\;.
  \end{split}
\end{equation}
Combining \eqref{eq_xie4hohngee0ooQuisahtoov} and the scalar Sobolev inequality \eqref{eq_AhR2ooxi0teixahheePoa5hu}, we get 
\begin{equation}
\label{eq_ThohS6aChuethe7phialarel}
\begin{split}
 \brk[\Big]{\int_{\Rset^n} \abs{\Deriv^{\ell} u}^\frac{2n}{n -2 (k - \ell)}}^{1 - \frac{(k - \ell)2}{n}}
  &\le 
  \Cl{cst_Wa5xoh3iab4aigeth0zeechi}\int_{\Rset^n}  \abs{A (\Deriv) u}^2
\end{split}
\end{equation}
and we have thus obtained the vector Sobolev inequality \eqref{eq_bah2acooqueBei4Ahph} when \(p = 2\) and the differential operator \(A(\Deriv)\) is injectively elliptic.

As a first example of naturally arising injectively elliptic constant coefficient homogeneous vector differential operator, we have the \(\operatorname{div}\)--\(\operatorname{curl}\) defined for each \(u \in C^\infty (\Rset^3, \Rset^3)\) as \(A (\Deriv) u \defeq (\operatorname{div} u, \operatorname{curl} u)\).
By Lagrange's identity for the dot and cross products, we have for every \(\xi \in \Rset^3\) and \(v \in \Rset^3\),
\begin{equation}
  \abs{A (\xi)[v]}^2 
  = \abs{(\xi \cdot v, \xi \times v)}^2
  = \abs{\xi \cdot v}^2 + \abs{\xi \times v}^2
  = \abs{\xi}^2 \abs{v}^2.
\end{equation}
The operator \(A(\Deriv)\) being injectively elliptic, we get as a particular case of \eqref{eq_ThohS6aChuethe7phialarel} for every \(u \in C^\infty_c(\Rset^3, \Rset^3)\) the Hodge--Sobolev estimate 
\begin{equation}
  \brk[\Big]{\int_{\Rset^3} \abs{u}^6}^{\frac{1}{3}}
  \le
  \C \int_{\Rset^n} \abs{\operatorname{div} u}^2 + \abs{\operatorname{curl} u}^2\;.
\end{equation}
More generaly, given \(m \in \set{1,\dotsc, n -1}\), we can consider the \emph{Hodge operator} \(A (\Deriv)\) defined for each \(\smash{u \in C^\infty(\Rset^n, \bigwedge^m \Rset^n)}\) by \(A (\Deriv) u \defeq (\dext u, \dext^* u)\) where \( \smash{\dext u\in C^\infty(\Rset^n, \bigwedge^{m+1} \Rset^n)}\) and \(\smash{\dext^*u\in C^\infty(\Rset^n, \bigwedge^{m-1} \Rset^n)}\) are respectively the exterior differential and codifferential of the differential form \(u\). 
The operator \(A (\Deriv)\) is also injectively elliptic, and we obtain in this case as a consequence of the inequality \eqref{eq_ThohS6aChuethe7phialarel} for every \(u \in C^\infty_c(\Rset^n, \bigwedge^m \Rset^n)\) the estimate
\begin{equation}
 \brk[\Big]{\int_{\Rset^n} \abs{u}^\frac{2n}{n - 2}}^{1 - \frac{2}{n}}
  \le 
  \C\int_{\Rset^n}  \abs{\dext u}^2 + \abs{\dext^* u}^2 \;.
\end{equation}

The \emph{symmetric derivative} \(\Deriv_{\mathrm{sym}}\) appears naturally in linear elasticity and is defined as the pointwise symmetric part of the derivative of a vector field \(u : \Rset^n\to \Rset^n\): \(\Deriv_{\mathrm{sym}} u(x) \defeq (\Deriv u (x) + (\Deriv u(x))^*)/2\).
The differential operator \(A (\Deriv) \defeq \Deriv_{\mathrm{sym}}\) is injectively elliptic. Indeed, one has for every \(\xi \in \Rset^n\) and \(v \in \Rset^n\), 
\begin{equation}
  \abs{A (\xi)[v]}^2
  = \frac{\abs{\xi \otimes v + v \otimes \xi}^2}{4}
  = \frac{\abs{\xi}^2 \abs{v}^2 + \abs{\xi \cdot v}^2}{2},
\end{equation}
which vanishes if and only if either \(\xi = 0\) or \(v = 0\).
As a consequence of \eqref{eq_ThohS6aChuethe7phialarel} we get for every \(u \in C^\infty_c(\Rset^n, \Rset^n)\) the Korn--Sobolev inequality 
\begin{equation}
 \brk[\Big]{\int_{\Rset^n} \abs{u}^\frac{2n}{n - 2}}^{1 - \frac{2}{n}}
  \le 
  \C\int_{\Rset^n}  \abs{\Deriv_{\operatorname{sym}} u}^2 \;.
\end{equation}

\subsection{\texorpdfstring{\(L^p\)}{Lᵖ} estimates}
In order to get the estimate \eqref{eq_ThohS6aChuethe7phialarel} for \(p \ne 2\), we first note that when the differential operator \(A (\Deriv)\) is injectively elliptic, the Fourier transform of \(u \in C^\infty_c (\Rset^n, V)\) can be written for every \(\xi \in \Rset^n \setminus \set{0}\) as 
\begin{equation}
\label{eq_Aqu1shoh0xu8ipesip0thook}
 \mathcal{F} u (\xi)
 = \brk{2 \pi i}^{-k} A (\xi)^\dagger[\brk{2 \pi i}^k A (\xi)[\mathcal{F} u (\xi)]]
 =  \brk{2 \pi i}^{-k} A (\xi)^\dagger[\mathcal{F} (A (\Deriv)u) (\xi)]\;,
\end{equation}
in view of \eqref{eq_du0zoovuociuBi9nu}.
In \eqref{eq_Aqu1shoh0xu8ipesip0thook}, \(A (\xi)^\dagger \in \Lin (E, V)\) denotes for each \(\xi \in \Rset^n \setminus \set{0}\) the Moore--Penrose generalized inverse of \(A (\xi)\), which can be computed as
\begin{equation}
\label{eq_Eib4Oel5Queew9Roeth3la3o}
  A (\xi)^\dagger = \brk[\big]{A (\xi)^* A (\xi)}^{-1} A (\xi)^*
\end{equation}
since \(A (\xi)\) is injective and thus \(A(\xi)^* A (\xi)\) is also;
the generalized inverse \(A (\xi)^\dagger\) satisfies \(A (\xi)^\dagger A (\xi) = \operatorname{id}_V\).
(The alert reader might have noted that the definition of \(A (\xi)^\dagger\) depends on the choice of a Euclidean structure on the linear spaces \(V\) and \(E\); this arbitrary inner product main role is to fix the choice of a left-inverse so that it depends smoothly on \(\xi\); otherwise the precise choice does not matter.)

As a consequence of \eqref{eq_Aqu1shoh0xu8ipesip0thook}, we have for every \(\xi \in \Rset^n \setminus \set{0}\) 
\begin{equation}
 \mathcal{F} (\Deriv^k u) (\xi)
 = \xi^{\otimes k} \otimes A (\xi)^\dagger[\mathcal{F} (A (\Deriv)u) (\xi)]\;.
\end{equation}
Since the mapping \(\xi \in \Rset^n \setminus \set{0} \mapsto A (\xi)^\dagger \in \Lin (E, V)\) is smooth and homogeneous of degree \(0\), 
by a classical multiplier theorem (see for example \cite{Stein_1970}*{ch.\ IV th.\ 3}), we get when \(p \in \intvo{1}{\infty}\) for every \(u \in C^\infty_c (\Rset^n, V)\) the estimate 
\begin{equation} 
\label{eq_chausaewagheich6eabizioS}
\int_{\Rset^n} \abs{\Deriv^{k} u}^p
  \le \C \int_{\Rset^n} \abs{A (\Deriv) u }^p \;,
\end{equation}
which follows directly from \eqref{eq_xie4hohngee0ooQuisahtoov} when \(p = 2\).
Combining the Sobolev inequality \eqref{eq_AhR2ooxi0teixahheePoa5hu} with \eqref{eq_chausaewagheich6eabizioS}, we get for \(p \in \intvo{1}{\frac{n}{k - \ell}}\) the vector Sobolev estimate
\begin{equation}
\label{eq_ohwe9Hojew5xeeMe5eiRoopa}
   \brk[\Big]{\int_{\Rset^n} \abs{\Deriv^{\ell} u}^\frac{np}{n - (k - \ell)p}}^{1 - \frac{(k - \ell)p}{n}}
  \le 
  \C
  \int_{\Rset^n} \abs{A (\Deriv)[u]}^p\;.
\end{equation}
As a consequence of \eqref{eq_ohwe9Hojew5xeeMe5eiRoopa} we get Hodge--Sobolev inequalities: for \(p \in \intvo{1}{3}\) and for every  \(u \in C^\infty_c (\Rset^3, \Rset^3)\) with \(n = 3\),
\begin{equation}
\label{eq_weiyag3iyein4fahF8jeeRei}
 \brk[\Big]{\int_{\Rset^3} \abs{u}^\frac{3p}{3 - p}}^{1 - \frac{p}{3}}
  \le 
  \C\int_{\Rset^3}  \abs{\operatorname{div} u}^p + \abs{\operatorname{curl} u}^p\,;
\end{equation}
and for \(p \in \intvo{1}{n}\) and for every \(u \in C^\infty_c (\Rset^n, \bigwedge^m \Rset^n)\),
\begin{equation}
\label{eq_dai0Tee4ieH1ienohf2eezie}
 \brk[\Big]{\int_{\Rset^n} \abs{u}^\frac{np}{n - p}}^{1 - \frac{p}{n}}
  \le 
  \C\int_{\Rset^n}  \abs{\dext u}^p + \abs{\dext^* u}^p\;,
\end{equation}
and a Korn--Sobolev inequality: for \(p \in \intvo{1}{n}\) and  for every \(u \in C^\infty_c (\Rset^n, \Rset^n)\),
\begin{equation}
\label{eq_aishiedae2quae0siu0Ev6ac}
 \brk[\Big]{\int_{\Rset^n} \abs{u}^\frac{np}{n - p}}^{1 - \frac{p}{n}}
  \le 
  \C\int_{\Rset^n}  \abs{\Deriv_{\mathrm{sym}} u}^p\;.
\end{equation}

\subsection{\texorpdfstring{\(L^1\)}{L¹} non-estimates}
In the endpoint \(p = 1\), the techniques behind the multiplier theorem that lead to \eqref{eq_chausaewagheich6eabizioS} merely yield a weak-type inequality: for every \(u \in C^\infty_c (\Rset^n, V)\) and  for every \(t \in \intvo{0}{\infty}\)
\begin{equation} 
\label{eq_otei6eiNgoh5yei2Bo9zahgi}
\mathcal{L}^n \brk[\big]{\set{x \in \Rset^n \st  \abs{\Deriv^{k} u(x)} \ge t}}
  \le \frac{\C}{t} \int_{\Rset^n} \abs{A (\Deriv) u} \;.
\end{equation}
The estimate \eqref{eq_otei6eiNgoh5yei2Bo9zahgi} is insufficient to get a Sobolev inequality of the form \eqref{eq_bah2acooqueBei4Ahph}. (Actually, the inequality \eqref{eq_chausaewagheich6eabizioS} for \(1 < p < 2\) is proved by Marcinkiewicz interpolation between the estimates \eqref{eq_xie4hohngee0ooQuisahtoov} and \eqref{eq_otei6eiNgoh5yei2Bo9zahgi}.)

The inequality \eqref{eq_otei6eiNgoh5yei2Bo9zahgi} cannot be improved into a strong type inequality. Donald S.\ \familyname{Ornstein} has proved that there is no nontrivial estimate \cite{Ornstein_1962}:
If \(B (\Deriv)\) is another homogeneous constant coefficient differential operator from \(V\) to a linear space \(F\) of order \(k\) and if for every \(u \in C^\infty_c (\Rset^n, V)\) the estimate 
\begin{equation}
\label{eq_Eexiew6iequacahbahgaewei}
\int_{\Rset^n} \abs{B (\Deriv) u} \le \C \int_{\Rset^n} \abs{A (\Deriv)u}\;,
\end{equation}
holds, then \(B (\Deriv) = L A (\Deriv)\), where \(L \in \Lin (E, F)\) is a constant coefficient linear mapping.
In other words, the estimate \eqref{eq_Eexiew6iequacahbahgaewei} can only hold for algebraic reasons, when the derivatives on its left-hand side are linear combinations of those appearing on its right-hand side, and never for analytic ones.
The proof of this non-estimate is done through convex integration; although Ornstein's original paper does not cover explicitly the vector case, it seems to be working in that case and more recent approaches do explicitly \citelist{\cite{Kirchheim_Kristensen_2016}\cite{Kirchheim_Kristensen_2011}} (for nonestimates through convex integration see also \cite{Conti_Faraco_Maggi_2005}). 
Interestingly, Ornstein proved his non-estimate in order to construct a distribution which is not a measure but whose first-order derivatives are of order \(1\), providing a counter-example to another question also raised in Laurent \familyname{Schwartz}'s \emph{Théorie des distributions} \citelist{\cite{Schwartz_1950}\cite{Schwartz_1951}}.

\section{The cancelling condition}

\subsection{A menagerie of inequalities}
\label{section_menagerie}

Even if, as we have seen above, Ornstein's non-estimate prevents us from obtaining a Sobolev inequality of the form \eqref{eq_bah2acooqueBei4Ahph} at the endpoint \(p = 1\), one might still hope to obtain directly such an inequality.

An endpoint Korn--Sobolev inequality generalizing the scalar Sobolev inequality \eqref{eq_aishiedae2quae0siu0Ev6ac} was indeed obtained for each \(u \in C^\infty( \Rset^n, \Rset^n)\) by Monty J. \familyname{Strauss} \cite{Strauss_1973}
\begin{equation}
\label{eq_aeth2Fai8OoSeik2eirukei1}
  \brk*{\int_{\Rset^n} \abs{u}^\frac{n}{n- 1}}^{1 - \frac{1}{n}}
  \le \C \int_{\Rset^n} \abs{\Deriv_{\mathrm{sym}} u}\;,
\end{equation}
with a proof similar to Gagliardo and Nirenberg's original proof of \eqref{eq_AhR2ooxi0teixahheePoa5hu}.

On the the other hand, the endpoint Hodge--Sobolev inequality \eqref{eq_weiyag3iyein4fahF8jeeRei}
\begin{equation}
\label{eq_mie9veikoe6EedaGoot3eaSu}
 \brk[\Big]{\int_{\Rset^3} \abs{u}^\frac{3}{2}}^{\frac{2}{3}}
  \le 
  \Cl{cst_geupheeta4uxahpohK9eHohc}\int_{\Rset^3}  \abs{\operatorname{div} u} + \abs{\operatorname{curl} u}
\end{equation}
\emph{does not hold} for every \(u \in C^\infty_c (\Rset^3, \Rset^3)\). 
In order to see this, one essentially takes \(u\) in \eqref{eq_mie9veikoe6EedaGoot3eaSu} to be a suitable regularization of the function \(x \mapsto x/\abs{x}^3\) whose divergence is a Dirac measure and whose curl vanishes.
Alternatively, taking \(u = \nabla v\), one can see that \eqref{eq_mie9veikoe6EedaGoot3eaSu} would imply 
an estimate for every \(v \in C^\infty_c (\Rset^3, \Rset)\),
\begin{equation}
  \brk[\Big]{\int_{\Rset^3} \abs{\nabla v}^\frac{3}{2}}^{\frac{2}{3}}
  \le 
  \Cr{cst_geupheeta4uxahpohK9eHohc} \int_{\Rset^3}  \abs{\Delta v}\;,
\end{equation}
which does not hold.
Quite surprisingly, Jean \familyname{Bourgain} and Haïm \familyname{Brezis}
have proved that the estimate \eqref{eq_mie9veikoe6EedaGoot3eaSu} holds for every vector field \(u \in C^\infty_c (\Rset^3, \Rset^3)\) for which \(\operatorname{div} u = 0\) \citelist{\cite{Bourgain_Brezis_2004}*{th.\ 2}\cite{Bourgain_Brezis_2007}*{cor.\ 12}}.
Moreover, they proved that if \(m \in \set{2, \dotsc, n - 2}\) --- so that in particular, \(n \ge 4\) --- one has the endpoint Hodge--Sobolev inequality: for every \(u \in C^\infty_c (\Rset^n, \bigwedge^m \Rset^n)\) 
\begin{equation}
\label{eq_phuaSh5ahdango4ohvoochae}
\vspace{-.2ex}
 \brk[\Big]{\int_{\Rset^n} \abs{u}^\frac{n}{n - 1}}^{1 - \frac{1}{n}}
  \le 
  \C\int_{\Rset^n}  \abs{\dext u} + \abs{\dext^* u}\;,
\vspace{-.2ex}
\end{equation}
extending \eqref{eq_dai0Tee4ieH1ienohf2eezie}
(see also \cite{Lanzani_Stein_2005}).

\subsection{Towards cancelling operators}

In order to identify the differential operators \(A (\Deriv)\) for which the Sobolev inequality \eqref{eq_ohwe9Hojew5xeeMe5eiRoopa} holds in the endpoint case \(p = 1\), we examine carefully the counterexample to the Hodge--Sobolev inequality \eqref{eq_mie9veikoe6EedaGoot3eaSu}.
The essential ingredient is the construction of \(u\) such that \(A (\Deriv) u = (\operatorname{div} u, \operatorname{curl} u)\) is a Dirac measure.
In the general case we want to find some  \(u\) such that 
\begin{equation}
\label{eq_ahfigh4ooPoh6eix7ohnii5Y}
 A (\Deriv) u = e \delta_0\;,
\end{equation}
for some fixed vector \(e \in E\).
Passing to the Fourier transform in \eqref{eq_ahfigh4ooPoh6eix7ohnii5Y}, 
we get 
\begin{equation}
\label{eq_Zahgh4Oob1ilaeyoh0shah4N}
 \brk{2 \pi i}^k A (\xi)[\mathcal{F} u (\xi)]= e\;.
\end{equation}
In order for \eqref{eq_Zahgh4Oob1ilaeyoh0shah4N} to have a solution, the differential operator \(A(D)\) and the vector \(e\) should satisfy the condition
\begin{equation}
\label{eq_pheeZahng1kiB9jiega7ohj3}
 e \in \bigcap_{\xi \in \Rset^n \setminus \set{0}} A (\xi)[V];
\end{equation}
in particular, \eqref{eq_ahfigh4ooPoh6eix7ohnii5Y} will not have a solution for some non-trivial \(e\) when the right-hand side of \eqref{eq_pheeZahng1kiB9jiega7ohj3} reduces to \(\set{0}\).
This motivates the following definition \cite{VanSchaftingen_2013}. 
\begin{definition}
\label{definition_cancelling}
The homogeneous differential operator \(A (\Deriv)\) is \emph{cancelling} whenever 
\[
 \bigcap_{\xi \in \Rset^n \setminus \set{0}} A (\xi)[V] = \set{0}\;.
\]
\end{definition}

Thanks to the injective ellipticity of \(A (\Deriv)\), for every \(\xi \in \Rset^n \setminus \set{0}\), \(A (\xi)^\dagger\) is well-defined through \eqref{eq_Eib4Oel5Queew9Roeth3la3o} and the function \(\xi \in \Rset^n \setminus \set{0} \mapsto A (\xi)^\dagger \in \Lin(E, V)\) is homogeneous of degree \(-k\); hence through the extension of the homogeneous distribution \(A^\dagger\) from \(\Rset^n \setminus \set{0}\) to a homogeneous distribution on \(\Rset^n\) \cite{Hormander_1990_I}*{th.\ 3.2.3 \& 3.2.4}, the preservation of the smoothness in the extension and its temperate character \cite{Hormander_1990_I}*{th.\ 7.1.18},
and the homogeneity of the resulting Fourier transform \cite{Hormander_1990_I}*{th.\ 7.1.16}, we get the following construction (see \citelist{\cite{Bousquet_VanSchaftingen_2014}*{lem.\ 2.1}\cite{Raita_2019}} for a proof based on these ingredients and \cite{VanSchaftingen_2023}*{prop.\ 2} for a direct self-contained proof).

\begin{proposition}
\label{proposition_fundamental_solution}
If \(A (\Deriv)\) is injectively elliptic, then there exists \(G_A \in C^\infty (\Rset^n \setminus \set{0}, \Lin (E, V))\) such that 
\begin{enumerate}[label=(\roman*)]
 \item \(\mathcal{F} G_A (\xi)  = \brk*{2 \pi i}^{-k} A (\xi)^\dagger\),
 \item \label{it_Exiemaephu9Ex3teizaiziet}
 for every \(x \in \Rset^n \setminus \set{0}\) and every \(t \in \Rset \setminus \set{0}\), 
    \[
    G_A (t x) = t^{k - n} \brk[\big]{G_A (x) - \ln \abs{t} \, P_A (x)}\;,
    \]
    where the function \(P_A : \Rset^n \to \Lin (E, V)\) is a homogeneous polynomial of degree \(k - n\) when \(k \ge n\) and \(k - n\) is even, and is \(0\) otherwise.
\end{enumerate}
\end{proposition}

If we assume that the operator \(A (\Deriv)\) is injectively elliptic, then taking some \(e \in E\) and performing a suitable regularization of \(G_A [e] \in C^\infty (\Rset^n \setminus \set{0}, V)\), we get the necessity of the cancellation in the following theorem \cite{VanSchaftingen_2013}*{} (see also \cite{VanSchaftingen_2023}*{th. 9}).

\begin{theorem}
\label{theorem_cancelling_necessary}
If \(A (\Deriv)\) is injectively elliptic, then for every \(u \in C^\infty_c (\Rset^n, V)\),   
\begin{equation}
\label{eq_yiey3Yud2aeB8eimahd0baiy}
\brk[\Big]{\int_{\Rset^n} \abs{\Deriv^{\ell} u}^\frac{n}{n - (k - \ell)}}^{1 - \frac{k - \ell}{n}}
  \le 
  \C
  \int_{\Rset^n} \abs{A (\Deriv)[u]}\;,
\end{equation}
if and only if \(A (\Deriv)\) is cancelling.
\end{theorem}

\Cref{theorem_cancelling_necessary} assumes somehow implicitly very strong boundary conditions on the function \(u\), through the compact support assumption. 
The corresponding theory for weaker boundary conditions introduces additional restrictions on the admissible class of operators and boundary conditions \citelist{\cite{Brezis_VanSchaftingen_2007}\cite{Gmeineder_Raita_2019}\cite{Gmeineder_Raita_VanSchaftingen}
}.

We end this section by noting that the case in which the inequalities in \cref{section_menagerie} corresponding exactly to cancelling operators, illustrating the necessity of the cancellation in \cref{theorem_cancelling_necessary} and supporting the sufficiency of this condition for endpoint Sobolev inequalities.

We begin with the \(\operatorname{div}\)--\(\operatorname{curl}\) operator. 
We have then for every \(\xi \in \Rset^3\) and \(v \in \Rset^3\), 
\(A (\xi)[v] = (\xi \cdot v, \xi \times v)\). 
Hence for every \(\xi \in \Rset^3\),
\begin{equation}
 A (\xi)[\Rset^3] = \set{(\xi \cdot v, \xi \times v) \in \Rset \times \Rset^3 \st v \in \Rset^3 }
 = \Rset \times \set{\xi}^\perp\;,
\end{equation}
so that 
\begin{equation}
 \bigcap_{\xi \in \Rset^3\setminus \set{0}}
 A (\xi)[\Rset^3]
 = 
 \bigcap_{\xi \in \Rset^3\setminus \set{0}}
 \Rset \times \set{\xi}^\perp
 = \Rset \times \set{0}\;,
\end{equation}
which proves that the operator \(A (\Deriv)\) is not cancelling; 
thanks to \cref{theorem_cancelling_necessary} we recover the fact that the endpoint Hodge--Sobolev inequality \eqref{eq_mie9veikoe6EedaGoot3eaSu} does not hold.

If \(m \in \set{2, \dotsc, n - 2}\), the Hodge operator \(A (\Deriv)\) defined for each \(u \in C^\infty (\Rset^n, \bigwedge^m \Rset^n)\) by \(A (\Deriv) u = (\dext u, \dext^* u)\) is cancelling.
Indeed, for every \(\xi \in \Rset^n\) and \(v \in \bigwedge^m \Rset^n\), we have \(A (\xi)[v] = (\xi \wedge v, \xi \lrcorner v) \), and thus for every \(\xi \in \Rset^n\),
\begin{equation}
\begin{split}
 A (\xi)[V]
 &= \set{(\xi \wedge v, \xi \lrcorner v) \in {\textstyle \bigwedge^{m + 1}} \Rset^n \times {\textstyle \bigwedge^{m - 1}} \Rset^n \st v \in   {\textstyle \bigwedge^{m}} \Rset^n}\\
&\subseteq \set[\big] { (w, w_*) \in {\textstyle \bigwedge^{m + 1}} \Rset^n \times {\textstyle \bigwedge^{m - 1}} \Rset^n \st \xi \wedge w = 0 \text{ and } \xi \lrcorner w_* = 0}\;, 
 \end{split}
\end{equation}
so that 
\begin{equation}
\bigcap_{\xi \in \Rset^n \setminus \set{0}}
A (\xi)[V] = 0\;,
\end{equation}
and the operator \(A (\Deriv)\) is cancelling; this could have also been seen as a consequence of \cref{theorem_cancelling_necessary} and of the estimate \eqref{eq_phuaSh5ahdango4ohvoochae}.

Finally, the symmetric derivative operator \(A (\Deriv)\) defined for every \(u \in C^\infty (\Rset^n, \Rset^n)\) by
\(
  A (\Deriv)u \defeq \Deriv_{\mathrm{sym}} u
\) is cancelling if and only if \(n \ge 2\).
Indeed, for every \(\xi, v, w \in \Rset^n\), we have
\[
  A (\xi)[v]:(w \otimes w) = \frac{(\xi \otimes v + v \otimes \xi):(w \otimes w)}{2}= (\xi \cdot w) (v\cdot w)\;.
\]
Therefore,
\[
 A (\xi)[\Rset^n]
 \subseteq \set{e \in \Lin^2_{\mathrm{sym}} (\Rset^n, \Rset) \st \text{for every \(w \in \xi^\perp\), } e:(w \otimes w) = 0},
\]
and hence 
\[
   \bigcap_{\xi \in \Rset^n \setminus \set{0}}
    A (\xi)[\Rset^n]
    = \set{0}\;,
\]
so that the operator \(A (\Deriv)\) is cancelling, which is consistent with \cref{theorem_cancelling_necessary} and the Korn--Sobolev inequality \eqref{eq_aeth2Fai8OoSeik2eirukei1}.

\section{Compatibility conditions and the duality estimate}

\subsection{Higher integrability through duality estimates}
In order to get \eqref{eq_yiey3Yud2aeB8eimahd0baiy}, besides using the Sobolev inequality \eqref{eq_AhR2ooxi0teixahheePoa5hu}, we could also control an \(L^q\)--norm of \(\Deriv^{\ell} u\) through a multiplier theorem. 
Indeed, if we have 
\begin{equation}
\label{eq_loogiez3ic8Vethohph4ieha}
 A (\Deriv) u = B (\Deriv) g,
\end{equation}
for some differential operator \(B (\Deriv)\) of order \(k - \ell\) on \(\Rset^n\) from \(F\) to \(E\) 
and some function \(g \in C^\infty (\Rset^n, F)\), then we can rewrite \eqref{eq_loogiez3ic8Vethohph4ieha} through a Fourier transform in the frequency domain as 
\begin{equation}
 \brk{2 \pi i}^k A (\xi)[\mathcal{F}u (\xi)] = \brk{2 \pi i}^{k - \ell} B (\xi)[\mathcal{F} g (\xi)]
\end{equation}
and thus 
\begin{equation}
 \mathcal{F} (\Deriv^{\ell} u) (\xi)
 = 
 \xi^{\otimes \ell} A(\xi)^\dagger [B (\xi)[\mathcal{F}g (\xi)]];
\end{equation}
since the map \(\xi \in \Rset^n \setminus \set{0} \to\xi^{\otimes \ell} A(\xi)^\dagger B (\xi) \in \Lin(F, \Lin^{\ell}_{\mathrm{sym}} (\Rset^n, V))\) is smooth and homogeneous of degree \(0\), by the same classical  multiplier theorem as we used for \eqref{eq_chausaewagheich6eabizioS}
(see for example \cite{Stein_1970}*{ch.\ IV th.\ 3}), we get when \(q \in \intvo{1}{\infty}\) the estimate 
\begin{equation} 
\label{eq_veeseeseik4dih9ohwooghiK}
\int_{\Rset^n} \abs{\Deriv^{\ell} u}^q
  \le \C \int_{\Rset^n} \abs{g}^q \;.
\end{equation}
So one way to get the estimate \eqref{eq_bah2acooqueBei4Ahph}, would be to construct a function \(g : \Rset^n \to F\) that satisfies the identity \eqref{eq_loogiez3ic8Vethohph4ieha} and an estimate
\begin{equation}
\label{eq_zei5ovahy1Raiposh6kauhae}
  \brk[\Big]{\int_{\Rset^n} \abs{g}^\frac{n}{n - (k - \ell)}}^{1 - \frac{k - \ell}{n}} \le \C 
  \int_{\Rset^n} \abs{A (\Deriv) u}\;
\end{equation}
and apply then the inequality \eqref{eq_veeseeseik4dih9ohwooghiK} with \(q = \frac{n}{n - (k - \ell)}\).

By duality, the existence of \(g\) satisfying estimate \eqref{eq_zei5ovahy1Raiposh6kauhae} with \eqref{eq_loogiez3ic8Vethohph4ieha} is equivalent to have for every \(\varphi \in C^\infty_c (\Rset^n, E)\)
\begin{equation}
\label{eq_Oosheequoom7eege3thou8Ee}
  \abs[\Big]{\int_{\Rset^n} (A(\Deriv) u) \cdot \varphi\,}
  \le 
  \C 
  \int_{\Rset^n} \abs{A (\Deriv) u} \,\brk[\Big]{\int_{\Rset^n} \abs{\Deriv^{k - \ell} \varphi}^\frac{n}{k - \ell} }^\frac{k - \ell}{n};
\end{equation}
One cannot hope in general to have the estimate
\begin{equation}
\label{eq_Ii3ieng9maiyeerohyuigaeP}
  \abs[\Big]{\int_{\Rset^n} f \cdot \varphi\,}
  \le 
  \Cl{cst_Ahph5oodie7eiPh5eib3iwa8}
  \int_{\Rset^n} \abs{f} \, \brk[\Big]{\int_{\Rset^n} \abs{\Deriv^{k - \ell} \varphi}^n }^\frac{k - \ell}{n},
\end{equation}
because taking \(f\) to be an approximation of the identity in \eqref{eq_Ii3ieng9maiyeerohyuigaeP} would yield the Sobolev-type estimate 
\begin{equation}
 \abs{\varphi (x)}\le \Cr{cst_Ahph5oodie7eiPh5eib3iwa8} \brk[\Big]{\int_{\Rset^n} \abs{\Deriv^{k - \ell} \varphi}^\frac{n}{k - \ell} }^\frac{k - \ell}{n},
\end{equation}
which is well-known to fail when \(0 < k - \ell < n\).

\subsection{Compatibility conditions}

The failure of the estimate \eqref{eq_Ii3ieng9maiyeerohyuigaeP} does not obliterate any hope of obtaining \eqref{eq_Oosheequoom7eege3thou8Ee}.
Indeed, the differential operator \(A (\Deriv)\) prescribes some structure to the function \(A(\Deriv) u\).
Such constraints are well-known for the derivative operator, whose image are curl-free vector fields: for \(u \in C^\infty (\Rset^n, \Rset)\), one has indeed for every \(i, j \in \set{1, \dotsc, n}\),
\begin{equation}
\label{eq_Re8quae2rii8vuom8kai2Ba3}
  \partial_j (\partial_i u) = \partial_i (\partial_j u)\;.
\end{equation}
Similarly for the \emph{Hodge complex,} if \(m \in \set{2, \dotsc, n - 2}\) we have for every differential form \(u \in C^\infty (\Rset^n, \bigwedge^m \Rset^n)\)
\begin{align}
\label{eq_eigoiSuthah1Phootee8uad1}
 \dext(\dext u) &= 0&
 &\text{ and }&
 \dext^*(\dext^* u) & = 0\;.
\end{align}
For the symmetric derivative \(\Deriv_{\mathrm{sym}}\) one has more elaborate classical \emph{Saint-Venant compatibility conditions} (see \citelist{\cite{Timoshenko_Goodier_1951}*{ch.\ 9}\cite{Ciarlet_2013}*{\S 6.18}}): for every \(u \in C^\infty (\Rset^n, \Rset^n)\)  and every  \(i, j, k, \ell \in \set{1, \dotsc, n}\), one has
\begin{equation}
\label{eq_Xah5quaejieRahche9roova8}
 \partial_{k\ell} (\partial_i u^j +\partial_j u^i)
 +\partial_{ij} (\partial_k u^\ell +\partial_\ell u^k)
 = \partial_{kj} (\partial_i u^\ell +\partial_\ell u^i)
 +\partial_{i\ell} (\partial_k u^j +\partial_j u^k)\;.
\end{equation}

In fact, if the differential operator \(A(\Deriv)\) is injectively elliptic, one can always construct a differential operator \(L(\Deriv)\) that gives the associated \emph{compatibility conditions}.
Indeed, setting 
\begin{equation}
\label{eq_xooch1Ocui9jechohnaeFo2I}
 L (\xi) \defeq \det \bigl(A (\xi)^*A (\xi)\bigr)\,
 \brk[\big]{\operatorname{id}_E\, -\, A (\xi) (A (\xi)^* A (\xi))^{-1} A (\xi)^*}\;,
\end{equation}
we observe that for every \(\xi \in \Rset^n \setminus \set{0}\),
\[
\ker L (\xi) = \set[\big]{ e \in E \st A (\xi)  (A (\xi)^* A (\xi))^{-1} A (\xi)^* [e] = e}
= A (\xi)[V]\;;
\]
moreover \(A (\xi)^*  A (\xi)\) is a polynomial in \(\xi\), so that 
\(\det (A (\xi)^*  A (\xi)) (A (\xi)^*  A (\xi))^{-1}\) is also a polynomial in \(\xi\) and thus \(L (\xi)\) is a polynomial in \(\xi\); we have thus proved the following proposition.

\begin{proposition}
 \label{proposition_compatibility_conditions}
If \(A (\Deriv)\) is injectively elliptic, then there exists a homogeneous differential operator \(L (\Deriv)\)
from \(E\) to \(F\) on \(\Rset^n\) such that for every \(\xi \in \Rset^n \setminus \set{0}\),
\begin{equation}
\label{eq_phicee9aad0ugh2thaeTueph}
  A (\xi)[V] = \ker L (\xi)\;.
\end{equation}
\end{proposition}

It should be noted that the order of the operator given by the identity \eqref{eq_xooch1Ocui9jechohnaeFo2I} is far from minimal, and is much higher than the one appearing in our specific examples \eqref{eq_Re8quae2rii8vuom8kai2Ba3}, \eqref{eq_eigoiSuthah1Phootee8uad1} and \eqref{eq_Xah5quaejieRahche9roova8}.

\subsection{Bourgain--Brezis duality estimate}

Motivated by the construction of an injectively elliptic operator's compatibility conditions in \cref{proposition_compatibility_conditions} and by the definition of cancelling operator in \cref{definition_cancelling}.

\begin{definition}
\label{definition_cocancelling}
The homogeneous differential operator \(L (\Deriv)\) is \emph{cocancelling} whenever 
\[
 \bigcap_{\xi \in \Rset^n \setminus \set{0}} \ker L (\xi) = \set{0}\;.
\]
\end{definition}

In particular, if the differential operator \(A (\Deriv)\) is injectively elliptic and if \(L (\Deriv)\) is the compatibility conditions differential operator given by \cref{proposition_compatibility_conditions}, then \(A (\Deriv)\) is cancelling if and only if \(L (\Deriv)\) is cocancelling.

The cocancellation condition is similar to conditions characterizing the dimension of measures lying in the kernel of differential operators \citelist{\cite{Roginskaya_Wojciechowski_2006}\cite{ArroyoRabasa_DePhilippis_Hirsch_Rindler_2019}\cite{DePhilippis_Rindler_2016}\cite{Arroyo_Rabasa_2020}}.

We are now going to prove that if \eqref{eq_Ii3ieng9maiyeerohyuigaeP} holds for every \(f \in C^\infty_c (\Rset^n, E)\), then \(L(\Deriv)\) is cocancelling.
Indeed, let \(e \in \bigcap_{\xi \in \Rset^n \setminus \set{0}} \ker L (\xi)\). For every \(g \in C^\infty_c (\Rset^n, \Rset)\), we have that 
\begin{equation}
 \mathcal{F} (L (\Deriv)(g e)) (\xi)
 =  \brk{2 \pi i}^m L(\xi)[e \mathcal{F} g (\xi)] = 0\;,
\end{equation}
where \(m \in \Nset \setminus \set{0}\) is the order of the differential operator \(L(\Deriv)\).
Hence, for every \(\varphi \in C^\infty_c (\Rset^n, \Rset)\), we have \(L(D)(ge)=0\) and thus by \eqref{eq_Ii3ieng9maiyeerohyuigaeP}
\begin{equation}
\begin{split}
 \abs{e}^2 \abs[\Big]{\int_{\Rset^n} g \varphi\,}
 &= \abs[\Big]{\int_{\Rset^n} (g e) \cdot (\varphi e)\,}\\
 &\le \Cl{cst_Da5aerah2keeroNa3Ohy3eiv} \int_{\Rset^n} \abs{g e}\,
 \brk[\Big]{\int_{\Rset^n} \abs{\Deriv (\varphi e) }^n }^{\frac{1}{n}} =\Cr{cst_Da5aerah2keeroNa3Ohy3eiv} \abs{e}^2 \int_{\Rset^n} \abs{g}\,
 \brk[\Big]{\int_{\Rset^n} \abs{\Deriv \varphi }^n }^{\frac{1}{n}}\;,
\end{split}
\end{equation} 
which cannot hold unless \(e = 0\).
This proves the necessity of the cocancellation in the following theorem due to the author \cite{VanSchaftingen_2013} (see also \citelist{\cite{Bourgain_Brezis_2004}\cite{Bourgain_Brezis_2007}\cite{VanSchaftingen_2004_div}\cite{VanSchaftingen_2004_ARB}\cite{VanSchaftingen_2008}\cite{VanSchaftingen_2013}})

\begin{theorem}
\label{theorem_cocancelling}
There exists a constant \(\Cl{cst_yohVaemeiwuphojoh3ohm3ga} \in \intvo{0}{\infty}\) such that
for every \(f \in C^\infty_c (\Rset^n, E)\) satisfying \(L (\Deriv) f = 0\) on \(\Rset^n\) and every \(\varphi \in C^\infty_c (\Rset^n, E)\), 
\begin{equation}
  \label{eq_koo4lui9Eipo7XeenaeghaoR}
  \abs[\Big]{\int_{\Rset^n} f \cdot \varphi\,}
  \le 
  \Cr{cst_yohVaemeiwuphojoh3ohm3ga} \int_{\Rset^n} \abs{f}\; \brk[\Big]{\int_{\Rset^n} \abs{\Deriv \varphi}^n}^\frac{1}{n}
\end{equation}
if and only if the operator \(L (\Deriv)\) is cocancelling.
\end{theorem}

\Cref{theorem_cocancelling} and Sobolev’s embedding imply \eqref{eq_Ii3ieng9maiyeerohyuigaeP} and sufficiency in \cref{theorem_cancelling_necessary}.

\Cref{theorem_cocancelling}  can be seen as a weaker replacement for the missing embedding of the critical Sobolev space \(W^{1, n} (\Rset^n)\) into \(L^\infty (\Rset^n)\).
Although many properties of \(W^{1, n} (\Rset^n)\) can be captured by its embedding into the space of functions of vanishing mean oscillation \(\mathrm{BMO} (\Rset^n)\), the estimate \eqref{eq_koo4lui9Eipo7XeenaeghaoR} can be shown to capture some stronger property of the space \(W^{1, n} (\Rset^n)\) \cite{VanSchaftingen_2006} (see also \cite{VanSchaftingen_2023}*{\S 5.1}).

Before going to a proof of the necessity, we first discuss some particular cases of \cref{theorem_cocancelling}. 

As a first example, we have the \(\operatorname{curl}\) operator acting on vector fields \(f : \Rset^n\to \Rset^n\) as  \(L (\Deriv) \defeq \operatorname{curl} f = D f - (Df)^*\).
We have for every \(\xi \in \Rset^n \setminus \set{0}\),
\begin{equation}
  \ker L (\xi) = 
  \set{e \in \Rset^n \st \xi \otimes e = e \otimes \xi}
  = \Rset \xi\;,
\end{equation}
and thus when \(n \ge 2\),
\begin{equation}
 \bigcap_{\xi \in \Rset^n \setminus \set{0}}\ker L (\xi)
 = \bigcap_{\xi \in \Rset^n \setminus \set{0}}\Rset \xi
 = \set{0}\;.
\end{equation}
In this case, the estimate \eqref{eq_koo4lui9Eipo7XeenaeghaoR} is essentially a dual formulation of the endpoint Sobolev inequality \eqref{eq_AhR2ooxi0teixahheePoa5hu}. 
Indeed, if \(\operatorname{curl} f = 0\), then \(f = \nabla u\) and thus
\begin{equation}
 \int_{\Rset^n} f \cdot \varphi
 = \int_{\Rset^n} \nabla u \cdot \varphi
 = - \int_{\Rset^n} u \operatorname{div} \varphi\;,
\end{equation}
so that by Hölder's inequality and the Sobolev inequality \eqref{eq_AhR2ooxi0teixahheePoa5hu} in the endpoint \(p = 1\)
\begin{equation}
\begin{split}
 \abs[\Big]{\int_{\Rset^n} f \cdot \varphi\,}
 &\le \brk[\Big]{\int_{\Rset^n} \abs{u}^\frac{n}{n - 1}}^{1 - \frac{1}{n}}
 \brk[\Big]{\int_{\Rset^n} \abs{\operatorname{div} \varphi}^n}^\frac{1}{n}\\
 &\le 
 \Cl{cst_uus5ohzauvukiy8meuh1aeGh} \brk[\Big]{\int_{\Rset^n} \abs{f}}
 \brk[\Big]{\int_{\Rset^n} \abs{\Deriv \varphi}^n}^\frac{1}{n}\;.
\end{split}
\end{equation}

A second important case is when \(L (\Deriv)\) is the divergence operator. 
Indeed, we have then for every \(\xi \in \Rset^n\),
\begin{equation*}
 \ker L (\xi) = \set{e \in \Rset^n \st \xi \cdot e = 0}
 = \set{\xi}^\perp\;,
\end{equation*}
and thus when \(n \ge 2\)
\begin{equation*}
 \bigcap_{\xi \in \Rset^n \setminus \set{0}}
 \ker L (\xi) 
 = \bigcap_{\xi \in \Rset^n \setminus \set{0}}
 \set{\xi}^\perp
  =\set{0}\;.
\end{equation*}
Although when \(n =2\) curl-free and divergence-free vector fields are equivalent (up to an isometry in their target space), this is not anymore the case when \(n \ge 3\), and the inequality \eqref{eq_koo4lui9Eipo7XeenaeghaoR} is then stronger than the Sobolev inequality \eqref{eq_AhR2ooxi0teixahheePoa5hu} at the endpoint \(p = 1\).

We mentioned above that the Sobolev inequality \eqref{eq_AhR2ooxi0teixahheePoa5hu} at the endpoint \(p = 1\) was in some sense equivalent to the classical isoperimetric inequality.
It turns out that the estimate \eqref{eq_koo4lui9Eipo7XeenaeghaoR} for divergence-free vector fields also has some geometric flavour. 
By a classical approximation argument, \eqref{eq_koo4lui9Eipo7XeenaeghaoR} holds also for divergence-free measures and thus in particular for circulation integrals along closed curves. One thus reaches that if \(\Gamma\) is closed curve with tangent vector \(t\) and length \(\abs{\Gamma}\)
\begin{equation}
\label{eq_tughiet3laeGhei6chiequim}
 \abs[\Big]{\int_{\Gamma} \varphi \cdot t}
 \le \C \abs{\Gamma} \,\brk[\Big]{\int_{\Rset^n} \abs{\Deriv \varphi}^n}^\frac{1}{n}\;;
\end{equation}
the estimate on circulation integrals \eqref{eq_tughiet3laeGhei6chiequim} was obtained by Jean \familyname{Bourgain}, Haïm \familyname{Brezis} and Petru \familyname{Mironescu} \cite{Bourgain_Brezis_Mironescu_2004} (see also \cite{VanSchaftingen_2004_circ}).
Whereas when \(n = 2\), \eqref{eq_tughiet3laeGhei6chiequim} can be obtained as a direct consequence of Green's theorem, the isoperimetric inequality and the Cauchy--Schwarz inequality, with an optimal constant, this is not anymore the case in higher dimensions \(n \ge 3\).
In view of the highly geometric character of \eqref{eq_tughiet3laeGhei6chiequim}, there are several interesting open questions on the optimal constant and optimizers for \eqref{eq_tughiet3laeGhei6chiequim} when \(n \ge 3\) \cite{Brezis_VanSchaftingen_2008}.
Using a result of Stanislav K.\ \familyname{Smirnov} on the decomposition of divergence-free vector fields into elementary solenoids \cite{Smirnov_1993}, one can deduce \eqref{eq_koo4lui9Eipo7XeenaeghaoR} for divergence-free vector fields from \eqref{eq_tughiet3laeGhei6chiequim}, with the same constant \cite{Bourgain_Brezis_2004}.

The essential ingredient of Jean \familyname{Bourgain} and Haïm \familyname{Brezis}’s original proof of the estimate of \cref{theorem_cocancelling} is an approximation property for critical Sobolev functions: For every \(\varepsilon \in \intvo{0}{+\infty}\), they show the existence of a constant \(M_\varepsilon \in \intvo{0}{\infty}\) such that for every function \(u \in \dot{W}^{1, n} (\Rset^n, \Rset)\), they can construct through a Littlewood--Paley decomposition a function \(\smash v \in (\dot{W}^{1, n} \cap L^\infty)(\Rset^n, \Rset)\) satisfying the estimates
\begin{equation*}
  \norm{\Deriv v}_{L^n (\Rset^n)} + \norm{v}_{L^\infty (\Rset^n)} \le M_\varepsilon \norm{\Deriv u}_{L^n(\Rset^n)}
\end{equation*}
and 
\begin{equation*}
  \norm{\Deriv' u - \Deriv'v}_{L^n (\Rset^n)} \le \varepsilon \norm{\Deriv u}_{L^n (\Rset^n)}\;,
\end{equation*}
where \(\Deriv'\) denotes the derivative with respect to the \(n -1\) first variables.
This approximation result generalized a previous approximation result that they had obtained in their work on the divergence equation \citelist{\cite{Bourgain_Brezis_2002}\cite{Bourgain_Brezis_2003}}; the original version covered the divergence operator \cite{Bourgain_Brezis_2007}*{th.\ 1$'$} and the operator \(L (\Deriv)[v_1, \dotsc, v_n]\defeq \partial_1^k v_1 + \dotsb + \partial_n^k v_n\) \cite{Bourgain_Brezis_2007}*{cor.\ 24}; suitable algebraic arguments can be used to extend their result to the full class of cocancelling operators \citelist{\cite{VanSchaftingen_2008}\cite{VanSchaftingen_2013}}.
The advantage of Bourgain and Brezis's proof compared to the more elementary approach presented in the next section is that it yields stronger estimates of the form 
\begin{equation}
  \abs[\Big]{\int_{\Rset^n} f \cdot \varphi}
  \le \C \norm{f}_{L^1 (\Rset^n) + \dot{W}^{-1,n/(n-1)}(\Rset^n)}
  \brk[\Big]{\int_{\Rset^n} \abs{\Deriv \varphi}^n}^\frac{1}{n}
\end{equation}
instead of \eqref{eq_koo4lui9Eipo7XeenaeghaoR}.

\section{Proving the duality estimate}

We present a proof of \cref{theorem_cocancelling} due to the author \citelist{\cite{VanSchaftingen_2008}\cite{VanSchaftingen_2013}} (see also \cite{VanSchaftingen_2023}).

\subsection{The case of the divergence}
We first prove the estimate \eqref{eq_koo4lui9Eipo7XeenaeghaoR} when \(L (\Deriv)\) is the divergence, following \cite{VanSchaftingen_2004_div}.
Fixing some vector \(\nu \in \Rset^n\), we are going to estimate
\begin{equation}
\label{eq_AehaeZ4dieh7ichax4ieyaex}
  \int_{\Rset^n} f \cdot \nu \,\phi\;,
\end{equation}
for \(\phi \in C^\infty_c (\Rset^n, \Rset)\).
Without loss of generality we can assume that \(\nu = e_n\) is the \(n\)--th vector in the canonical basis of \(\Rset^n\).
We write by Fubini's theorem
\begin{equation}
\label{eq_Ohseequaz4Faing1ingeeshi}
  \int_{\Rset^n} f \cdot e_n\, \phi
  = \int_{\Rset} \brk[\Big]{\int_{\Rset^{n - 1}} f  (\cdot, x_n) \cdot e_n \,\phi (\cdot, x_n)} \dif x_n\;.
\end{equation}
In order to estimate the inner integral on the right-hand side of \eqref{eq_Ohseequaz4Faing1ingeeshi}, we first note that for every \(\psi \in C^\infty (\Rset^{n - 1}, \Rset)\), we have 
\begin{equation}
\label{eq_tie1Lefemieyum2zasoeWaip}
 \abs[\Big]{\int_{\Rset^{n - 1}} f  (\cdot, x_n) \cdot e_n \,\psi} 
 \le\norm{\psi}_{L^\infty (\Rset^{n - 1})} \int_{\Rset^{n - 1}} \abs{f(\cdot, x_n)} \;.
\end{equation}
On the other hand, by the Gauss--Ostrogradsky divergence theorem, since 
\(\operatorname{div} f = 0\), we also have
\begin{equation}
\label{eq_hee5oob8ohhachookool6ahS}
 \int_{\Rset^{n - 1}} f  (\cdot, x_n) \cdot e_n \,\psi
 = - \int_{\Rset^{n} \times \intvo{x_n}{\infty}} \operatorname{div} (f\,\Psi)
 = - \int_{\Rset^{n} \times \intvo{x_n}{\infty}} f \cdot  \nabla \Psi\;,
\end{equation}
where the function \(\Psi : \Rset^n \to \Rset\) is defined by \(\Psi (x', x_n) \defeq \psi (x')\),
and thus it follows from \eqref{eq_hee5oob8ohhachookool6ahS} that 
\begin{equation}
\label{eq_eiqu4daeFuDeiv0ZaiPuijai}
  \abs[\Big]{\int_{\Rset^{n - 1}} f  (\cdot, x_n) \cdot e_n \,\psi}
  \le \norm{\Deriv \Psi}_{L^\infty (\Rset^{n})} \int_{\Rset^n} \abs{f} 
  = \norm{\Deriv \psi}_{L^\infty (\Rset^{n - 1})} \int_{\Rset^n} \abs{f} \;.
\end{equation}

Given \(\alpha \in \intvo{0}{1}\) and fixing a function \(\eta \in C^\infty_c (\Rset^{n - 1}, \Rset)\) such that \(\int_{\Rset^{n - 1}} \eta = 1\), we define for each \(\lambda \in \intvo{0}{\infty}\), the function \(\psi_{\lambda} : \Rset^{n - 1} \to \Rset\) by setting for each \(x \in \Rset^{n - 1}\)
\begin{equation}
\label{eq_eingaiYah0Eejie8ooQui9ni}
  \psi_\lambda (x) 
  \defeq 
  \int_{\Rset^{n - 1}} \eta (y)\,\psi (x - \lambda y) \dif y= 
  \int_{\Rset^{n - 1}} \eta \brk[\Big]{\frac{x}{\lambda} - z} \psi (\lambda z) \dif z\;.
\end{equation}
Since \(\int_{\Rset^{n - 1}} \eta = 1\), for each \(x \in \Rset^{n - 1}\) we have by \eqref{eq_eingaiYah0Eejie8ooQui9ni}
\begin{equation}
\label{eq_ieghaifahnge3eeSeiyiequ7}
    \psi_\lambda (x) - \psi (x) 
  = 
    \int_{\Rset^{n - 1}} \eta (y)\, \bigl(\psi (x - \lambda y) - \psi (x)\bigr) \dif y\;, 
\end{equation}
and thus by \eqref{eq_ieghaifahnge3eeSeiyiequ7},
\begin{equation}
  \label{eq_ahx1Ucheik4heGohx}
\begin{split}
\abs{\psi_\lambda (x) - \psi (x)} 
&\le \int_{\Rset^{n - 1}} \abs{\eta (y)} \, \abs{\psi (x - \lambda y) - \psi (x)} \dif y\\
&\le \seminorm{\psi}_{C^{0, \alpha} (\Rset^{n - 1})} \lambda^\alpha 
\int_{\Rset^{n - 1}} \abs{\eta (y)} \, \abs{y}^\alpha \dif y  = \Cl{cst_nooSu1eyohreiy8ei}   \seminorm{\psi}_{C^{0, \alpha} (\Rset^{n - 1})} \lambda^\alpha\; ,
\end{split}
\end{equation}
with the Hölder seminorm being defined as 
\begin{equation}
\label{eq_As1eeshaiquah1Oom2zah9ai}
 \abs{\psi}_{C^{0, \alpha} (\Rset^{n - 1})}
 \defeq \sup_{\substack{x, y \in \Rset^{n - 1}\\ x \ne y}}
 \frac{\abs{\psi (x) - \psi (y)}}{\abs{x - y}^\alpha}\;.
\end{equation}
On the other hand, for each \(x \in \Rset^{n - 1}\), we also have by \eqref{eq_eingaiYah0Eejie8ooQui9ni} 
\begin{equation}
  \label{eq_ka7IcieLie5Isaifa7xieshu}
\begin{split}
\Deriv \psi_\lambda (x)
&= 
\frac{1}{\lambda} \int_{\Rset^{n - 1}} \Deriv \eta \Bigl(\frac{x}{\lambda} - z\Bigr) \psi (\lambda z) \dif z
\\
&=
\frac{1}{\lambda} \int_{\Rset^{n - 1}} \Deriv \eta (y)\, \psi (x - \lambda y) \dif y  
= \frac{1}{\lambda} \int_{\Rset^{n - 1}} \Deriv \eta (y) \, \brk[\big]{\psi (x - \lambda y) - \psi (x)}\dif y\;,
\end{split}
\end{equation}
since \(\int_{\Rset^{n - 1}} \Deriv \eta = 0\).
Hence, we have by \eqref{eq_ka7IcieLie5Isaifa7xieshu} and by \eqref{eq_As1eeshaiquah1Oom2zah9ai}
\begin{equation}
  \label{eq_Ahdai9gareT3phein}
  \begin{split}
    \abs{\Deriv \psi_\lambda (x)}
    &\le
    \frac{1}{\lambda} 
    \int_{\Rset^{n - 1}} 
    \abs{\Deriv \eta (y)}\, \abs{\psi (x - \lambda y) - \psi (x)} \dif y\\
    &\le 
    \frac{\seminorm{\psi}_{C^{0, \alpha} (\Rset^{n - 1})}}{\lambda^{1 - \alpha}}
      \int_{\Rset^{n - 1}} \abs{\Deriv \eta (y)} \,\abs{y}^\alpha \dif y
      = 
      \Cl{cst_iehuje0xei8mong1W} \frac{\seminorm{\psi}_{C^{0, \alpha} (\Rset^{n - 1})}}{\lambda^{1 - \alpha}}\;.
  \end{split}
\end{equation}
By \eqref{eq_tie1Lefemieyum2zasoeWaip}, \eqref{eq_eiqu4daeFuDeiv0ZaiPuijai},  \eqref{eq_ahx1Ucheik4heGohx} and \eqref{eq_Ahdai9gareT3phein}, we have 
\begin{equation}
  \label{eq_maeNg8quaeho8ohDaye2ohVa}
\begin{split}
\abs*{\int_{\Rset^{n - 1}} f  (\cdot, x_n) \cdot e_n \,\psi}
&\le
\abs*{\int_{\Rset^{n - 1}} f  (\cdot, x_n) \cdot e_n (\psi - \psi_\lambda)}
+ \abs*{\int_{\Rset^{n - 1}} f  (\cdot, x_n) \cdot e_n \psi_\lambda}\\
&\le 
\norm{\psi - \psi_\lambda}_{L^\infty (\Rset^{n - 1})}  \int_{\Rset^{n - 1}} \abs{f(\cdot, x_n)} 
+ \norm{\Deriv \psi_\lambda}_{L^\infty (\Rset^{n - 1})} \int_{\Rset^n} \abs{f}\\
&\le 
\brk[\Big]{
\Cr{cst_nooSu1eyohreiy8ei}   \lambda^\alpha \int_{\Rset^{n - 1}} \abs{f(\cdot, x_n)} 
+ 
 \frac{\Cr{cst_iehuje0xei8mong1W}}{\lambda^{1 - \alpha}}\int_{\Rset^n} \abs{f} }\seminorm{\psi}_{C^{0, \alpha} (\Rset^{n - 1})}\;.
\end{split}
\end{equation}
If we choose 
\begin{equation*}
 \lambda \defeq \frac{\displaystyle \int_{\Rset^n} \abs{f}}{\displaystyle \int_{\Rset^{n - 1}} \abs{f(\cdot, x_n)}}\;,
\end{equation*}
we get from \eqref{eq_maeNg8quaeho8ohDaye2ohVa} the interpolation estimate
\begin{equation}
\label{eq_iFaenai7iengiej4soo7aubo}
  \abs[\Big]{\int_{\Rset^{n - 1}} f  (\cdot, x_n) \cdot e_n \,\psi}
  \le \C \brk[\Big]{\int_{\Rset^{n - 1}} \abs{f(\cdot, x_n)}}^{1 - \alpha} \brk[\Big]{\int_{\Rset^n} \abs{f}}^\alpha  \abs{\psi}_{C^{0, \alpha} (\Rset^{n - 1})}\;,
\end{equation}

By the Morrey--Sobolev inequality, we also have
\begin{equation}
\label{eq_deish2Xahjiepugha5soonad}
 \abs{\psi}_{C^{0,1/n} (\Rset^{n - 1})}
 \le 
\C
\brk*{\int_{\Rset^{n - 1}} \abs{\Deriv \psi}^n}^\frac{1}{n}\;.
\end{equation}
Combining \eqref{eq_iFaenai7iengiej4soo7aubo} with \(\alpha = \frac{1}{n}\) and \eqref{eq_deish2Xahjiepugha5soonad}, we get 
\begin{equation}
\label{eq_eipiu2waeTh8ooboophio5te}
  \abs[\Big]{\int_{\Rset^{n - 1}} f  (\cdot, x_n) \cdot e_n\, \psi}
  \le \Cl{cst_Ahgh1GewophaFoh1laiCuiph} \brk[\Big]{\int_{\Rset^{n - 1}} \abs{f(\cdot, x_n)}}^{1 - \frac{1}{n}} \brk[\Big]{\int_{\Rset^n} \abs{f}}^\frac{1}{n} \brk[\Big]{\int_{\Rset^{n - 1}} \abs{\Deriv \psi}^n}^\frac{1}{n}\;.
\end{equation}
Next, by \eqref{eq_Ohseequaz4Faing1ingeeshi} and \eqref{eq_eipiu2waeTh8ooboophio5te}, we obtain by Hölder's inequality
\begin{equation}
\label{eq_ieBaecheinohth1ohza1OhQu}
\begin{split}
 \abs[\Big]{\int_{\Rset^{n}} f \cdot e_n\, \phi}
 &\le \Cr{cst_Ahgh1GewophaFoh1laiCuiph}
 \brk[\Big]{\int_{\Rset^n} \abs{f}}^\frac{1}{n}
 \int_{\Rset} \brk[\Big]{\int_{\Rset^{n - 1}} \abs{f(\cdot, x_n)}}^{1 - \frac{1}{n}} \brk[\Big]{\int_{\Rset^{n - 1}} \abs{\Deriv \phi (\cdot, x_n)}^n}^\frac{1}{n} \dif x_n\\
 &\le \Cr{cst_Ahgh1GewophaFoh1laiCuiph}
 \int_{\Rset^n} \abs{f} \brk[\Big]{\int_{\Rset^{n}} \abs{\Deriv \phi}^n}^\frac{1}{n} \;.
\end{split}
\end{equation}
Therefore for any vector \(\nu \in \Rset^n\), it follows from \eqref{eq_ieBaecheinohth1ohza1OhQu} that we have proved
\begin{equation}
\label{eq_woNgoo3heiD5Oogh8eneeguo}
 \abs*{\int_{\Rset^{n}} f \cdot \nu \,\phi}
 \le \Cr{cst_Ahgh1GewophaFoh1laiCuiph}\abs{\nu}
 \brk*{\int_{\Rset^n} \abs{f}} \brk*{\int_{\Rset^{n - 1}} \abs{\Deriv \phi}^n}^\frac{1}{n} .
\end{equation}
Decomposing \(\varphi = \sum_{j = 1}^n e_i \phi_i\) with \(\phi_i \defeq e_i \cdot \varphi\), we finally obtain the announced inequality \eqref{eq_koo4lui9Eipo7XeenaeghaoR} from \eqref{eq_woNgoo3heiD5Oogh8eneeguo}.

\subsection{The general case}
We now consider the case where \(L (\Deriv)\) is a general cancelling operator.
Similarly to the estimate of \eqref{eq_AehaeZ4dieh7ichax4ieyaex} when \(L (\Deriv)\) is the divergence, we will estimate the integral
\begin{equation}
  \int_{\Rset^n} L(\Deriv)[\nu] \phi\;,
\end{equation}
for \(\nu \in \Rset^n\) and \(\phi \in C^\infty_c (\Rset^n, F)\).

First we readily have as a counterpart of \eqref{eq_tie1Lefemieyum2zasoeWaip}
\begin{equation}
\label{eq_uif9Eipo3Ia0Oi0booVeh3ph}
\begin{split}
 \abs[\Big]{\int_{\Rset^{n - 1}} L(e_n)[f  (\cdot, x_n)] \cdot \psi}
 &\le \norm{\psi}_{L^\infty (\Rset^{n - 1})} \int_{\Rset^{n - 1}} \abs{L(e_n)[f(\cdot, x_n)]}\\
 &\le \C \norm{\psi}_{L^\infty (\Rset^{n - 1})} \int_{\Rset^{n - 1}} \abs{f(\cdot, x_n)} \;.
 \end{split}
\end{equation}
Fixing \(x_n \in \Rset\), we define the function \(\theta_m : \Rset^n_+ \to \Rset\) for \(x' \in \Rset^{n - 1}\) and \(t \in \Rset\) by 
\begin{equation}
  \label{eq_GaechahLae3wai4ooYaik7cu}
\theta_m (x', x_n + t) \defeq \frac{t{}^{m - 1}}{(m - 1) !}\; 
\end{equation}
and \(\theta_0 \defeq 0\).
By \(k\) successive integration by parts we have, 
\begin{multline}
  \label{eq_uf4Koo1wahhic8aiBoo}
\int_{\Rset^{n - 1} \times \intvo{x_n}{\infty}} (\partial_n^m L(e_n)[f])\cdot (\theta_m \Psi)
- (-1)^{m} L (e_n)[f] \cdot (\partial_n^m (\theta_m \Psi))\\
  = - \sum_{j = 0}^{m - 1}
  (-1)^j
  \int_{\Rset^{n - 1} \times \set{x_n}} (\partial_n^{m - 1 - j} L(e_n)[f]) \cdot (\partial_n^j (\theta_m \Psi))\;,
\end{multline}
where \(\Psi \in C^\infty (\Rset^n, F)\) satisfies \(\Psi \vert_{\Rset^{n - 1} \times \set{x_n}} = \psi\).
For each \(j \in \set{0, \dotsc, m}\), we have by the Leibniz rule
\begin{equation}
  \label{eq_aiZootheiX1Dooche7ooqua7}
\partial_n^j (\theta_m \Psi)
= \sum_{i = 0}^j {\textstyle \binom{j}{i}} \, \theta_{m - i} \,\partial_n^{j - i} \Psi \;,
\end{equation}
since \(\partial_n^i \theta_m = \theta_{m - i}\).
It follows from  \eqref{eq_aiZootheiX1Dooche7ooqua7} that 
\(
\partial_n^j (\theta_m \Psi) = 0
\) on \(\Rset^{n - 1} \times \set{x_n}\) when \(j \in \set{1, \dotsc,  m - 2}\) whereas \(
\partial_n^{m - 1} (\theta_m \Psi)(\cdot, 0) = \psi\) on \(\Rset^{n - 1} \times \set{x_n}\). 
We can thus rewrite the right-hand side in \eqref{eq_uf4Koo1wahhic8aiBoo} as
\begin{equation}
  \sum_{j = 0}^{m - 1}
  (-1)^j
  \int_{\Rset^{n - 1} \times \set{x_n}} \hspace{-2em}(\partial_n^{m - j} L (e_n)[f]) \cdot (\partial_n^j (\theta_m \Psi))
  = (-1)^{m - 1} \int_{\Rset^{n - 1} \times \set{x_n}} (L(e_n)[f])\cdot \psi \;,
\end{equation}
and hence by \eqref{eq_uf4Koo1wahhic8aiBoo},
\begin{equation}
  \label{eq_viifuBaibee7Eephiyais7ph}
  \int_{\Rset^{n - 1} \times \set{x_n}} L(e_n)[f] \cdot \psi
  = \int_{\Rset^{n} \times \intvo{x_n}{\infty}} (-1)^{m}(\partial_n^m L(e_n)[f])\cdot (\theta_m \Psi) - L(e_n)[f]\cdot (\partial_n^m (\theta_m \Psi))\;.
\end{equation}
We also have by \eqref{eq_aiZootheiX1Dooche7ooqua7}, 
\begin{equation}
  \label{eq_HaiHahzou6ha6Hai5ay}
\abs[\Big]{
  \int_{\Rset^{n - 1} \times \intvo{x_n}{\infty}} (L(e_n)[f]) \cdot (\partial_n^m (\theta_m \Psi))}
\le 
\C \brk[\Big]{\sum_{j = 1}^m  \norm{\theta_j \Deriv^j \Psi}_{L^\infty(\Rset^{n - 1}\times (x_n,\infty))}}
\int_{\Rset^n} \abs{L (e_n)[f]}\;.
\end{equation}
Next, we rewrite the operator \(L (\xi)\) for \(\xi =( \xi', \xi_n) \in \Rset^n\)  as 
\begin{equation}
\label{eq_lah1mahphahJohleeThathai}
L(\xi) = \sum_{j = 0}^{m} \xi_n^{m - j} L_j (\xi')\;,
\end{equation}
where for every \(j \in \set{0, \dotsc, m}\), \(L_j(\Deriv')\) is a homogeneous linear differential operator on \(\Rset^{n - 1}\) of degree \(j\) from \(E\) to \(F\).
By our assumption we have \(L (\Deriv) f = 0\) and by \eqref{eq_lah1mahphahJohleeThathai} we have, since \(L (\nu) = L_0 (\xi)\),  
\begin{equation}
  \label{eq_maisiesequ7yuGhahch}
    \int_{\Rset^{n - 1} \times \intvo{x_n}{\infty}} (\partial_n^m L(\nu)[f])\cdot (\theta_m \Psi)
  =
  -  \sum_{j = 1}^{m } \int_{\Rset^{n - 1} \times \intvo{x_n}{\infty}} (\partial_n^{m - j}  L_j (\Deriv') [f])\cdot(\theta_m \Psi)\;.
\end{equation}
By integration by parts, we compute for every \(j \in \set{1, \dotsc, m}\),
\begin{equation}
  \label{eq_ciex5Ahh0aephooKaid}
\int_{\Rset^{n - 1} \times \intvo{x_n}{\infty}} (\partial_n^{m - j}  L_j (\Deriv') [f])\cdot (\theta_m \Psi)
= (-1)^{m} \int_{\Rset^{n - 1} \times \intvo{x_n}{\infty}} f \cdot L_j^* (\Deriv')\partial_n^{m - j} ( \theta_m \Psi)\;,
\end{equation}
where 
\(L_j (\xi')^* \in \Lin (E, F)\) is the adjoint to \(L_j (\xi')\).
We compute, for each \(j \in \set{1, \dotsc, m}\), by the general Leibniz rule again
\begin{equation}
\label{eq_Ahhaish2ieC5zeeMohyeelae}
\begin{split}
L_j (\Deriv')^*
\partial_n^{m - j} ( \theta_m \Psi)
&= 
\partial_n^{m - j}  (\theta_m L_j (\Deriv')^* \Psi)\\
&= \sum_{i= 0}^{m - j} \tbinom{m - j}{i} (\partial_n^i \theta_m) (\partial_n^{m - j - i} L_j (\Deriv')^* \Psi)\\
&= \sum_{i = 0}^{m - j} \tbinom{m - j}{i} \theta_{m - i} \partial_n^{m - j - i} L_j (\Deriv')^* \Psi\;,
\end{split}
\end{equation}
so that  by \eqref{eq_maisiesequ7yuGhahch} and  \eqref{eq_ciex5Ahh0aephooKaid},
\begin{equation}
  \label{eq_ci0gahJahrij1ohphei}
\abs[\Big]{
  \int_{\Rset^{n - 1}\times \intvo{x_n}{\infty}} (\partial_n^m L(\nu)[f])\cdot(\theta_m \Psi)
}
\le \C \brk[\bigg]{\sum_{j = 1}^m \norm{\theta_j \Deriv^j \Psi}_{L^\infty(\Rset^{n -1}\times \intvo{x_n}{\infty})}} \int_{\Rset^n} \abs{f}\;.
\end{equation}
Cobmining the identity \eqref{eq_viifuBaibee7Eephiyais7ph} and the inequalities \eqref{eq_HaiHahzou6ha6Hai5ay} and \eqref{eq_ci0gahJahrij1ohphei}  we get 
\begin{equation}
\label{eq_eineic5aX1Gub1sie6ehiez8}
  \abs[\Big]{
  \int_{\Rset^{n - 1} \times \set{x_n} } (L(e_n)[f]) \cdot \psi
}
\le 
   \C 
   \brk[\Big]{
   \sum_{j = 1}^m \norm{\theta_j \Deriv^j \Psi}_{L^\infty(\Rset^{n - 1} \times \intvo{x_n}{\infty})}
   } \int_{\Rset^n} \abs{f}\;;
\end{equation}

It remains now to make a suitable choice of \(\Psi\). If we merely set \(\Psi(x', x_n) = \psi (x')\), we will not be have any control on \(\theta_j \Deriv^j \psi\) if \(j \ge 2\).
The solution is then to fix a function \(\eta \in C^1_c (\Rset^{n - 1}, \Rset)\) such that \(\int_{\Rset^{n - 1}} \eta = 1\) and to define the function \(\Psi : \Rset^{n - 1} \times \intvo{x_n}{\infty} \to F\)
for each \((x', t) \in \Rset^{n - 1} \times \intvo{0}{\infty}\) by
\begin{equation}
\label{eq_eNoo6iu8roh0ulieCepohFah}
  \Psi (x', x_n + t) \defeq
  \frac{1}{t^{n - 1}}
  \int_{\Rset^{n - 1}} \psi (y) \eta \bigl(\tfrac{x' - y}{t} \bigr) \dif y
  =\int_{\Rset^{n - 1}} \psi (x  - t z) \eta (z) \dif z
  \;,
\end{equation}
for which we have for every \(j \in \Nset \setminus \set{0}\)
\begin{equation}
  \label{eq_di4xeex8Yiekakei4Tei1eli}
  \abs{\Deriv^j \Psi (x', x_n + t)} \le \frac{\C \norm{\Deriv \psi}_{L^\infty (\Rset^{n- 1})}}{t^{j - 1}}\;.
\end{equation}
Combining \eqref{eq_eineic5aX1Gub1sie6ehiez8} and \eqref{eq_di4xeex8Yiekakei4Tei1eli}, we get 
\begin{equation}
\label{eq_ua7phoochon6Iey9aiviey0a}
   \abs[\Big]{
  \int_{\Rset^{n - 1} \times \set{x_n} } (L(e_n)[f]) \cdot \psi
}
\le 
   \C 
  \norm{\Deriv \psi}_{L^\infty (\Rset^{n - 1})} \int_{\Rset^n_+} \abs{f}\;.
\end{equation}
We proceed then as in the case where \(L (\Deriv) = \operatorname{div}\), first interpolating between \eqref{eq_uif9Eipo3Ia0Oi0booVeh3ph} and \eqref{eq_ua7phoochon6Iey9aiviey0a}, and then integrating and applying Hölder's inequality; we obtain thus the inequality
\begin{equation}
\label{eq_eSh9ka8xeef5ahcaingeiloh}
 \abs[\Big]{\int_{\Rset^n} L (\nu)[f] \cdot \phi\,}
 \le \C \abs{\nu} \int_{\Rset^n} \abs{f}\, \brk[\Big]{\int_{\Rset^n} \abs{\Deriv \phi}^n}^\frac{1}{n}.
\end{equation}

Finally, since the operator \(L (\Deriv)\) is cocancelling, there exist \(\nu_1, \dotsc, \nu_r \in \Rset^n\) such that 
\begin{equation}
 \bigcap_{i = 1}^r \ker L (\nu_i) = \set{ 0}\;.
\end{equation}
By a standard linear algebra reasoning, there exist linear mappings \(Q_1, \dotsc, Q_r \in \Lin (F, E)\) such that 
\begin{equation}
 \operatorname{id}_E = \sum_{i = 1}^r Q_i L (\nu_i)\;.
\end{equation}
Hence, we have 
\begin{equation}
\label{eq_TieHuph5ChaPheef1laec4hi}
\int_{\Rset^n} f \cdot \varphi
 =\sum_{i = 1}^r \int_{\Rset^n} Q_i L(\nu_i)[f] \cdot \varphi
 = \sum_{i = 1}^r \int_{\Rset^n} L(\nu_i)[f] \cdot Q_i^*[\varphi]\;,
\end{equation}
and it follows from \eqref{eq_eSh9ka8xeef5ahcaingeiloh} and \eqref{eq_TieHuph5ChaPheef1laec4hi} that 
\begin{equation}
\begin{split}
 \abs[\Big]{\int_{\Rset^n} f \cdot \varphi\,}
 &\le \C \sum_{i = 1}^n \abs{\nu_i} \int_{\Rset^n} \abs{f}\, \brk[\Big]{\int_{\Rset^n} \abs{\Deriv Q_i^*[\varphi]}^n}^\frac{1}{n}\\
 &\le \C \int_{\Rset^n} \abs{f}\, \brk[\Big]{\int_{\Rset^n} \abs{\Deriv \varphi}^n}^\frac{1}{n}\;,
\end{split}
\end{equation}
which proves \eqref{eq_koo4lui9Eipo7XeenaeghaoR}.
This proves the sufficiency of the cocancellation condition for \eqref{eq_koo4lui9Eipo7XeenaeghaoR} to hold in \cref{theorem_cocancelling}.

\subsection{Getting the estimate for cancelling operators}

As the operator \(A (\Deriv)\) is injectively elliptic, let \(L (\Deriv)\) be the compatibility conditions operator given by \cref{proposition_compatibility_conditions};
since \(A (\Deriv)\) is cancelling, \(L(\Deriv)\) is cocancelling. 
By \cref{theorem_cocancelling} and the representation of bounded linear functionals on Sobolev spaces, there exists \(\smash g \in L^\frac{n}{n - 1} (\Rset^n, E \otimes \Rset^n)\) such that 
\(\operatorname{div} g = A (\Deriv) u\) in the sense of distributions
and 
\[
\brk[\Big]{\int_{\Rset^n} \abs{g}^\frac{n}{n - 1}}^{1 - \frac{1}{n}}
\le \C \int_{\Rset^n} \abs{A (\Deriv) u}\;.
\]
We reach the conclusion when \(\ell = k - 1\) through a suitable regularization argument and through \eqref{eq_veeseeseik4dih9ohwooghiK}; the case \(1 < k - \ell < n\) follows from the classical Sobolev embedding. 

\subsection{Further results}
The techniques of proof of \cref{theorem_cocancelling} can be adapted to fractional settings, thanks to the Fubini property and the Morrey--Sobolev embedding in fractional spaces; if \(p \in \intvo{m}{\infty}\), one gets the estimate in critical Sobolev spaces \citelist{\cite{VanSchaftingen_2008}\cite{VanSchaftingen_2013}} (see also \cite{VanSchaftingen_2023})
\begin{equation}
\label{eq_leil8Eechu0eiJahxoo8nek1}
  \abs[\Big]{\int_{\Rset^n} f \cdot \varphi}
  \le 
  \C \int_{\Rset^n} \abs{f} \; \brk[\Big]{\int_{\Rset^n} \int_{\Rset^n} \frac{\abs{\varphi(y) - \varphi (x)}^p}{\abs{y - x}^{2n}}\dif y \dif x}^\frac{1}{p}
\end{equation} 
if and only if \(L (\Deriv)\) is cocancelling 
as a consequence of \eqref{eq_leil8Eechu0eiJahxoo8nek1}; using then multiplier theorems in fractional Sobolev spaces one gets if \(\smash{\frac{1 - s}{n} = 1 - \frac{1}{p}}\) the estimate 
\begin{equation}
\label{eq_thoh2gohz4phahquahvohSuy}
\brk[\Big]{\int_{\Rset^n} \int_{\Rset^n}  \frac{\abs{\Deriv^{k - 1} u (y) - \Deriv^{k - 1} u (x)}^p} {\abs{y -x}^{n + sp}} \dif y \dif x}^\frac{1}{p} 
\le 
C
\int_{\Rset^n} \abs{A (\Deriv)[u]}\;,
\end{equation}
if and only if \(A (\Deriv)\) is cancelling \cite{VanSchaftingen_2013} (see also \cite{VanSchaftingen_2023}).

The estimate \eqref{eq_thoh2gohz4phahquahvohSuy} yields in turn through classical embeddings estimates in Besov spaces, Triebel--Lizorkin spaces and Lorentz spaces \cite{VanSchaftingen_2013}; for the endpoint cases in these spaces we refer the reader to \citelist{\cite{Stolyarov_2020}\cite{Spector_VanSchaftingen_2019}\cite{Hernandez_Raita_Spector_2022}}.

Similar results for the hyperbolic plane \cite{Chanillo_VanSchaftingen_Yung_2017_Variations} and symmetric spaces of noncompact type \cite{Chanillo_VanSchaftingen_Yung_2017_Symmetric} through the use of suitable integral-geometric formulae.

The approach presented here can also be used to obtain endpoint estimate on stratified homogeneous groups \citelist{\cite{VanSchaftingen_Yung_2022}\cite{Chanillo_VanSchaftingen_2009}}.

\section{Hardy-type inequalities}
The cancellation condition is not only a necessary and sufficient condition for endpoint Sobolev inequalities, but also for example for Hardy inequalities.
We have indeed the next counterpart of \cref{theorem_cancelling_necessary}, which originates in an estimate of Vladimir Gilelevich \familyname{Maz\cprime{}ya} \cite{Mazya_2010} (see also \cite{Bousquet_Mironescu_2011}) and is due to Pierre \familyname{Bousquet} and the author \cite{Bousquet_VanSchaftingen_2014}.

\begin{theorem}
  \label{theorem_Hardy}
If \(A (\Deriv)\) is injectively elliptic, then for every \(u \in C^\infty_c (\Rset^n, V)\),   
\begin{equation}
\label{eq_ahHaichooxah6pha5thah6ie}
\int_{\Rset^n} \frac{\abs{\Deriv^\ell u(x)}}{\abs{x}^{k - \ell}} \dif x 
\le 
\C
\int_{\Rset^n} \abs{A (\Deriv)[u]}\;,
\end{equation}
if and only if \(A (\Deriv)\) is cancelling.
\end{theorem}

The main tool in the proof of \cref{theorem_Hardy} will be the following duality estimate, which is due to Pierre \familyname{Bousquet} and to the author \cite{Bousquet_VanSchaftingen_2014} (see also \cite{Raita_2019}).

\begin{proposition}
\label{proposition_L1_compensation}
If \(L (\Deriv)\) is a cocancelling operator of order \(m \in \Nset \setminus \set{0}\), if \(f \in C^\infty_c (\Rset^n, E)\) satisfies \(L (\Deriv) f = 0\) and if \(\varphi \in C^\infty_c (\Rset^n, E)\),
one has
  \begin{equation}
  \label{eq_xeiGh4eefai0xiephu2Choo7}
    \abs[\Big]{\int_{\Rset^n} f \cdot \varphi}
    \le 
    \C
    \sum_{j = 1}^m
     \int_{\Rset^n} \abs{f (x) } \,\abs{x}^{j} \, \abs{\Deriv^j \varphi (x)}\dif x\;.
  \end{equation}
\end{proposition}

\begin{proof}
Since the operator \(L (\Deriv)\) is cocancelling, there exist \(\xi_1, \dotsc, \xi_r \in \Rset^n \setminus \set{0}\) such that 
\(
 \bigcap_{i = 1}^r \ker L (\xi_i) = \set{ 0}
\).
By a standard linear algebra reasoning, there exist linear mappings \(Q_1, \dotsc, Q_r \in \Lin (F, E)\) such that 
\(
 \smash{\operatorname{id}_E = \sum_{i = 1}^r Q_i L (\xi_i)}
\).
Defining the polynomial \(P: \Rset^n \to \Lin (E, F)\) for each \(x \in \Rset^n\) by \(\smash{P (x) \defeq \sum_{i = 1}^r (\xi_i \cdot x)^m Q_i^* /m!}\), we have then \(\operatorname{id}_E = L (\Deriv)^*[P]\), and thus
\[
\int_{\Rset^n} f \cdot \varphi  = 
\int_{\Rset^n} {f}\cdot {L (\Deriv)^*[P] \varphi}.
\]
Since \(L(\Deriv)f = 0\), we compute by integration by parts
\[
\begin{split}
\int_{\Rset^n} f \cdot L (\Deriv)^*[P] \varphi
&= \int_{\Rset^n}  f \cdot L (\Deriv)^*[P] \varphi
- (-1)^m \int_{\Rset^n} L (\Deriv) f\cdot P \varphi\\
&=\int_{\Rset^n}  f \cdot (L (\Deriv)^*[P] \varphi
- L(\Deriv)^*[ P \varphi])\;,
\end{split}
\]
from which the conclusion \eqref{eq_xeiGh4eefai0xiephu2Choo7} follows.
\end{proof}

\begin{proof}[Proof of \cref{theorem_Hardy}]
Let \(G_A : \Rset^n \to \Lin (V, E)\) be the representation kernel given by \cref{proposition_fundamental_solution}, which exists since the operator \(A (\Deriv)\) is injectively elliptic.
To see the necessity of the cancellation, we take some \(e \in \bigcap_{\xi \in \Rset \setminus \set{0}} A (\xi)[V]\) and we perform a suitable regularization of \(G_A [e] \in C^\infty (\Rset^n \setminus \set{0}, V)\).

We assume now that the operator \(A(\Deriv)\) is cancelling. 
Choosing \(\varrho \in C^\infty_c(\Rset^n, \Rset)\) such that \(\varrho = 1\) on \(B_{1/4}(0)\) and \(\varrho=0\) on \(\Rset^n \setminus B_{1/2}(0)\), we define the kernels \(H_A : \Rset^n \times \Rset^n \to \Lin (V, E)\) and \(K_A : \Rset^n \times \Rset^n \to \Lin (V, E)\) for every \(x, y \in (\Rset^n \setminus \set{0})  \times \Rset^n \) with \(x \ne y\) by
\[
 H_A (x, y) \defeq \varrho \Bigl( \frac{y}{\abs{x}} \Bigr) \Deriv^{\ell} G_A (x)
\]
and
\[
 K_A (x, y) \defeq \Deriv^{\ell} G_A (x - y) - \varrho \Bigl( \frac{y}{\abs{x}} \Bigr) \Deriv^{\ell} G_A (x)\;.
\]
If \(L (\Deriv)\) are the compatibility conditions of order \(m\) given by \cref{proposition_compatibility_conditions}, we have that \(L (\Deriv)\) is cocancelling and that \(L (\Deriv) A (\Deriv)u = 0\).
By \cref{proposition_L1_compensation} and by the homogeneity of \(\Deriv^{\ell}_A G\)  (see \cref{proposition_fundamental_solution} \ref{it_Exiemaephu9Ex3teizaiziet}), we have
for every \(x \in \Rset^n \setminus \set{0}\), 
\begin{equation}
\begin{split}
\abs[\Big]{\int_{\Rset^n}  H_A (x, y) [A (\Deriv) u (y)] \dif y}
&\le \C \sum_{j = 1}^m \int_{B_{\abs{x}/2}(0)} \abs{A (\Deriv) u(y)} \abs{y}^j \abs{\Deriv_y^j H(x, y)} \dif y\\
&\le \C \sum_{j = 1}^m \int_{B_{\abs{x}/2}(0)} \frac{ \abs{A (\Deriv) u (y)} \abs{y}^j}{\abs{x}^{n - k + \ell + j}}  \dif y\\
&\le \C  \int_{B_{\abs{x}/2}(0)} \frac{ \abs{A (\Deriv) u (y)} \abs{y}}{\abs{x}^{n - k + \ell + 1}}  \dif y\;,
\end{split}
\end{equation}
and thus 
\begin{equation}
  \label{eq_yoh9Acohb7cah3Ohgoo7eish}
  \begin{split}
\int_{\Rset^n} &\abs[\Big]{\int_{\Rset^n}  H_A (x, y) [A (\Deriv) u (y)] \dif y}\frac{\mathrm{d} x}{\abs{x}^{k - \ell}} \\
&\le \Cl{cst_eiL2tahg3iigeidoozeeChei} \int_{\Rset^n} \int_{B_{\abs{x}/2}(0)} \frac{\abs{y}\abs{A (\Deriv) u (y)} }{\abs{x}^{n + 1}} \dif y \dif x\\
&=  \Cr{cst_eiL2tahg3iigeidoozeeChei} \int_{\Rset^n} \int_{\Rset^{n} \setminus B_{2 \abs{y}}(0)} \frac{\abs{y}\abs{A (\Deriv) u (y)} }{\abs{x}^{n + 1}} \dif x \dif y
\le \C \int_{\Rset^n} \abs{A (\Deriv) u}\;.
\end{split}
\end{equation}
Next we have 
\begin{equation}
  \label{eq_uFiel4aemou4aiNgizoow3ob}
\int_{\Rset^n} \abs[\Big]{\int_{\Rset^n}  K_A (x, y) [A (\Deriv) u (y)] \dif y} \frac{\mathrm{d} x}{\abs{x}^{k - \ell}} 
 \le  \int_{\Rset^n} \int_{\Rset^n}  \frac{\abs{K_A (x, y)}}{\abs{x}^{k - \ell}} \dif x \abs{A (\Deriv) u (y)}\dif y\;.
\end{equation}
Using again the homogeneity of \(\Deriv^\ell G_A\), we get if \(\abs{x} < 2 \abs{y}\)
\[
 \abs{K_A (x, y)} \le \frac{\C}{\abs{x-y}^{n - k + \ell} }\:,
\]
and if \(\abs{x} \ge 2 \abs{y}\)
\[
\abs{K_A (x, y)} \le  \frac{\C\abs{y}}{\abs{x}^{n - k + \ell + 1} }\;,
\]
so that, since \(k - \ell < n\),
\begin{equation}
  \label{eq_ne4dae2Ocai4Uh2ar1iuk4so}
  \begin{split}
 \int_{\Rset^n}  \frac{\abs{K_A (x, y)}}{\abs{x}^{k - \ell}} \dif x
&\le \C \brk[\bigg]{\int_{B_{2 \abs{y}}(0)} \frac{\mathrm{d} x}{\abs{x - y}^{n - k + \ell}\abs{x}^{k - \ell}}  + \int_{\Rset^n \setminus B_{2 \abs{y}}(0)}\!\! \ \frac{\abs{y}}{\abs{x}^{n + 1}} \dif x}\\
&\le \C\;.
\end{split}
\end{equation}
Combining \eqref{eq_uFiel4aemou4aiNgizoow3ob} and \eqref{eq_ne4dae2Ocai4Uh2ar1iuk4so}, we conclude that 
\begin{equation}
  \label{eq_datahlieghaivi0ahdooshaY}
  \int_{\Rset^n} \abs[\Big]{\int_{\Rset^n} K_A (x, y) [A (\Deriv) u (y)] \dif y} \frac{\mathrm{d} x}{\abs{x}^{k - \ell}}  \le \C \int_{\Rset^n} \abs{A (\Deriv) u }\;.
\end{equation}
The estimate \eqref{eq_ahHaichooxah6pha5thah6ie} then follows from \eqref{eq_yoh9Acohb7cah3Ohgoo7eish} and \eqref{eq_datahlieghaivi0ahdooshaY}.
\end{proof}

When \(k \ge n\), as a consequence of \cref{theorem_Hardy}, one gets the estimate 
\begin{equation}
\label{eq_cequae1eogheG4tohHohB9ae}
\norm{\Deriv^{n - k} u}_{L^\infty (\Rset^n)}
\le 
\C
\int_{\Rset^n} \abs{A (\Deriv)[u]}\;,
\end{equation}
which follows either from the boundedness of the resulting representation formula or directly from \eqref{eq_ahHaichooxah6pha5thah6ie} and the Sobolev representation formula \cite{Bousquet_VanSchaftingen_2014}*{th.\ 1.3}.
In fact, if the operator \(A (\Deriv)\) is injectively elliptic, Bogdan \familyname{Raiță} \cite{Raita_2019} has proved that the estimate \eqref{eq_cequae1eogheG4tohHohB9ae} is equivalent with the \emph{weak cancellation} property that for every \(e \in \bigcap_{\xi \in \Rset^n \setminus \set{0}} A (\xi)[V]\) one has
\begin{equation}
  \label{eq_aiboh2iesaeMisee2oyahah4}
  \int_{\Sset^{n - 1}} \xi^{\otimes k - n} A(\xi)^{-1} [e] \dif \xi = 0\;
\end{equation}
(for a proof, see also \cite{VanSchaftingen_2023}*{\S 5.4}).

The cancellation condition also appears in other settings such as the characterization of operators for which \(A(\Deriv) u\) being a measure implies the continuity of \(u\) \cite{Raita_Skorobogatova_2020}.


\begin{bibsection}
\begin{biblist} 

\bib{Arroyo_Rabasa_2020}{article}{
    author={Arroyo-Rabasa, Adolfo},
    title={An elementary approach to the dimension of measures satisfying a
    first-order linear PDE constraint},
    journal={Proc. Amer. Math. Soc.},
    volume={148},
    date={2020},
    number={1},
    pages={273--282},
    issn={0002-9939},
    doi={10.1090/proc/14732},
}
    \bib{ArroyoRabasa_DePhilippis_Hirsch_Rindler_2019}{article}{
        author={Arroyo-Rabasa, Adolfo},
        author={De Philippis, Guido},
        author={Hirsch, Jonas},
        author={Rindler, Filip},
        title={Dimensional estimates and rectifiability for measures satisfying
        linear PDE constraints},
        journal={Geom. Funct. Anal.},
        volume={29},
        date={2019},
        number={3},
        pages={639--658},
        issn={1016-443X},
        doi={10.1007/s00039-019-00497-1},
    }
\bib{Aubin_1976}{article}{
    author={Aubin, Thierry},
    title={Probl\`emes isop\'{e}rim\'{e}triques et espaces de Sobolev},
    journal={J. Differential Geometry},
    volume={11},
    date={1976},
    number={4},
    pages={573--598},
    issn={0022-040X},
}

\bib{Bourgain_Brezis_2002}{article}{
    author={Bourgain, Jean},
    author={Brezis, Ha\"{i}m},
    title={Sur l'\'{e}quation \(\operatorname{div}\,u=f\)},
    journal={C. R. Math. Acad. Sci. Paris},
    volume={334},
    date={2002},
    number={11},
    pages={973--976},
    issn={1631-073X},
}
    
\bib{Bourgain_Brezis_2003}{article}{
    author={Bourgain, Jean},
    author={Brezis, Ha\"{i}m},
    title={On the equation \(\operatorname{div}\, Y=f\) and application to control of phases},
    journal={J. Amer. Math. Soc.},
    volume={16},
    date={2003},
    number={2},
    pages={393--426},
    issn={0894-0347},
    doi={10.1090/S0894-0347-02-00411-3},
}

\bib{Bourgain_Brezis_2004}{article}{
    author={Bourgain, Jean},
    author={Brezis, Ha\"{i}m},
    title={New estimates for the Laplacian, the div--curl, and related Hodge systems},
    journal={C. R. Math. Acad. Sci. Paris},
    volume={338},
    date={2004},
    number={7},
    pages={539--543},
    issn={1631-073X},
    doi={10.1016/j.crma.2003.12.031},
}

\bib{Bourgain_Brezis_2007}{article}{
    author={Bourgain, Jean},
    author={Brezis, Ha\"{i}m},
    title={New estimates for elliptic equations and Hodge type systems},
    journal={J. Eur. Math. Soc. (JEMS)},
    volume={9},
    date={2007},
    number={2},
    pages={277--315},
    issn={1435-9855},
    doi={10.4171/JEMS/80},
}

\bib{Bourgain_Brezis_Mironescu_2004}{article}{
    author={Bourgain, Jean},
    author={Brezis, Haim},
    author={Mironescu, Petru},
    title={\(H^{1/2}\) maps with values into the circle: minimal connections, lifting, and the Ginzburg--Landau equation},
    journal={Publ. Math. Inst. Hautes \'{E}tudes Sci.},
    number={99},
    date={2004},
    pages={1--115},
    issn={0073-8301},
    doi={10.1007/s10240-004-0019-5},
}

\bib{Bousquet_Mironescu_2011}{article}{
   author={Bousquet, Pierre},
   author={Mironescu, Petru},
   title={An elementary proof of an inequality of Maz'ya involving $L^1$
   vector fields},
   conference={
      title={Nonlinear elliptic partial differential equations},
   },
   book={
      series={Contemp. Math.},
      volume={540},
      publisher={Amer. Math. Soc.}, 
      address={Providence, R.I.},
   },
   date={2011},
   pages={59--63},
   doi={10.1090/conm/540/10659},
}
   
\bib{Bousquet_VanSchaftingen_2014}{article}{
    author={Bousquet, Pierre},
    author={Van Schaftingen, Jean},
    title={Hardy--Sobolev inequalities for vector fields and canceling differential operators},
    journal={Indiana Univ. Math. J.},
    volume={63},
    date={2014},
    number={5},
    pages={1419--1445},
    issn={0022-2518},
    doi={10.1512/iumj.2014.63.5395},
}


\bib{Brezis_2011}{book}{
    author={Brezis, Ha{\"{\i}}m},
    title={Functional analysis, Sobolev spaces and partial differential
    equations},
    series={Universitext},
    publisher={Springer}, 
    address={New York},
    date={2011},
    pages={xiv+599},
    isbn={978-0-387-70913-0},
}

\bib{Brezis_VanSchaftingen_2007}{article}{
    author={Brezis, Ha\"{i}m},
    author={Van Schaftingen, Jean},
    title={Boundary estimates for elliptic systems with \(L^1\)--data},
    journal={Calc. Var. Partial Differential Equations},
    volume={30},
    date={2007},
    number={3},
    pages={369--388},
    issn={0944-2669},
    doi={10.1007/s00526-007-0094-9},
}
    
\bib{Brezis_VanSchaftingen_2008}{article}{
    author={Brezis, Ha\"{i}m},
    author={Van Schaftingen, Jean},
    title={Circulation integrals and critical Sobolev spaces: problems of
      optimal constants},
    conference={
      title={Perspectives in partial differential equations, harmonic
        analysis and applications},
    },
    book={
      series={Proc. Sympos. Pure Math.},
      volume={79},
      publisher={Amer. Math. Soc., Providence, R.I.},
    },
    date={2008},
    pages={33--47},
    doi={10.1090/pspum/079/2500488},
}    

\bib{Chanillo_VanSchaftingen_2009}{article}{
    author={Chanillo, Sagun},
    author={Van Schaftingen, Jean},
    title={Subelliptic Bourgain-Brezis estimates on groups},
    journal={Math. Res. Lett.},
    volume={16},
    date={2009},
    number={3},
    pages={487--501},
    issn={1073-2780},
    doi={10.4310/MRL.2009.v16.n3.a9},
}

\bib{Chanillo_VanSchaftingen_Yung_2017_Variations}{article}{
    author={Chanillo, Sagun},
    author={Van Schaftingen, Jean},
    author={Yung, Po-Lam},
    title={Variations on a proof of a borderline Bourgain--Brezis Sobolev
    embedding theorem},
    journal={Chinese Ann. Math. Ser. B},
    volume={38},
    date={2017},
    number={1},
    pages={235--252},
    issn={0252-9599},
    doi={10.1007/s11401-016-1069-y},
}

\bib{Chanillo_VanSchaftingen_Yung_2017_Symmetric}{article}{
    author={Chanillo, Sagun},
    author={Van Schaftingen, Jean},
    author={Yung, Po-Lam},
    title={Bourgain--Brezis inequalities on symmetric spaces of non-compact type},
    journal={J. Funct. Anal.},
    volume={273},
    date={2017},
    number={4},
    pages={1504--1547},
    issn={0022-1236},
    doi={10.1016/j.jfa.2017.05.005},
}
    
\bib{Ciarlet_2013}{book}{
    author={Ciarlet, Philippe G.},
    title={Linear and nonlinear functional analysis with applications},
    publisher={Society for Industrial and Applied Mathematics}, 
    address={Philadelphia, Pa.},
    date={2013},
    pages={xiv+832},
    isbn={978-1-611972-58-0},
} 

\bib{Conti_Faraco_Maggi_2005}{article}{
    author={Conti, Sergio},
    author={Faraco, Daniel},
    author={Maggi, Francesco},
    title={A new approach to counterexamples to $L^1$ estimates: Korn's inequality, geometric rigidity, and regularity for gradients of separately convex functions},
    journal={Arch. Ration. Mech. Anal.},
    volume={175},
    date={2005},
    number={2},
    pages={287--300},
    issn={0003-9527},
    doi={10.1007/s00205-004-0350-5},
}

\bib{DePhilippis_Rindler_2016}{article}{
    author={De Philippis, Guido},
    author={Rindler, Filip},
    title={On the structure of \(\mathcal{A}\)-free measures and applications},
    journal={Ann. of Math. (2)},
    volume={184},
    date={2016},
    number={3},
    pages={1017--1039},
    issn={0003-486X},
    doi={10.4007/annals.2016.184.3.10},
}

\bib{Federer_Fleming_1960}{article}{
    author={Federer, Herbert},
    author={Fleming, Wendell H.},
    title={Normal and integral currents},
    journal={Ann. of Math. (2)},
    volume={72},
    date={1960},
    pages={458--520},
    issn={0003-486X},
    doi={10.2307/1970227},
}

\bib{Gagliardo_1958}{article}{
    author={Gagliardo, Emilio},
    title={Propriet\`a di alcune classi di funzioni in pi\`u variabili},
    journal={Ricerche Mat.},
    volume={7},
    date={1958},
    pages={102--137},
}

\bib{Gmeineder_Raita_2019}{article}{
    author={Gmeineder, Franz},
    author={Rai\c{t}\u{a}, Bogdan},
    title={Embeddings for $\mathbb{A}$-weakly differentiable functions on
    domains},
    journal={J. Funct. Anal.},
    volume={277},
    date={2019},
    number={12},
    pages={108278, 33},
    issn={0022-1236},
    doi={10.1016/j.jfa.2019.108278},
}

\bib{Gmeineder_Raita_VanSchaftingen}{article}{
    author={Gmeineder, Franz},
    author={Rai\c{t}\u{a}, Bogdan},
    author={Van Schaftingen, Jean},
    title={Boundary ellipticity and limiting \(L^1\)-estimates on halfspaces},
    note={arXiv:2211.08167},
}

\bib{Hernandez_Raita_Spector_2022}{article}{
    author={Hernandez, F.}, 
    author={Raiță, B.},
    author={Spector, D.},
    title={Endpoint \(L^1\) estimates for Hodge systems},
    journal={Math. Ann.},
    date={2022},
    doi={10.1007/s00208-022-02383-y},
}

\bib{Hormander_1990_I}{book}{
    author={H{\"o}rmander, Lars},
    title={The analysis of linear partial differential operators},
    part={I},
    series={Grundlehren der Mathematischen Wissenschaften},
    volume={256},
    edition={2},
    subtitle={Distribution theory and Fourier analysis},
    publisher={Springer}, 
    address={Berlin},
    date={1990},
    pages={xii+440},
    isbn={3-540-52345-6},
    doi={10.1007/978-3-642-61497-2},
}

\bib{Kirchheim_Kristensen_2011}{article}{
    author={Kirchheim, Bernd},
    author={Kristensen, Jan},
    title={Automatic convexity of rank-1 convex functions},
    journal={C. R. Math. Acad. Sci. Paris},
    volume={349},
    date={2011},
    number={7-8},
    pages={407--409},
    issn={1631-073X},
    doi={10.1016/j.crma.2011.03.013},
}
    
\bib{Kirchheim_Kristensen_2016}{article}{
    author={Kirchheim, Bernd},
    author={Kristensen, Jan},
    title={On rank one convex functions that are homogeneous of degree one},
    journal={Arch. Ration. Mech. Anal.},
    volume={221},
    date={2016},
    number={1},
    pages={527--558},
    issn={0003-9527},
    doi={10.1007/s00205-016-0967-1},
}

\bib{Lanzani_Stein_2005}{article}{
    author={Lanzani, Loredana},
    author={Stein, Elias M.},
    title={A note on div curl inequalities},
    journal={Math. Res. Lett.},
    volume={12},
    date={2005},
    number={1},
    pages={57--61},
    issn={1073-2780},
    doi={10.4310/MRL.2005.v12.n1.a6},
}

\bib{Mazya_1960}{article}{
    author={Maz\cprime ya, Vladimir G.},
    title={Classes of domains and imbedding theorems for function spaces},
    journal={Soviet Math. Dokl.},
    volume={1},
    date={1960},
    pages={882--885},
    issn={0197-6788},
}

\bib{Mazya_2010}{article}{
    author={Maz\cprime{}ya, Vladimir G.},
    title={Estimates for differential operators of vector analysis involving \(L^1\)-norm},
    journal={J. Eur. Math. Soc. (JEMS)},
    volume={12},
    date={2010},
    number={1},
    pages={221--240},
    issn={1435-9855},
    doi={10.4171/JEMS/195},
}

\bib{Mazya_2011}{book}{
    author={Maz\cprime ya, Vladimir G.},
    title={Sobolev spaces with applications to elliptic partial differential equations},
    series={Grundlehren der mathematischen Wissenschaften},
    volume={342},
    edition={2},
    publisher={Springer}, 
    address={Heidelberg},
    date={2011},
    pages={xxviii+866},
    isbn={978-3-642-15563-5},
    doi={10.1007/978-3-642-15564-2},
}

\bib{Nirenberg_1959}{article}{
    author={Nirenberg, L.},
    title={On elliptic partial differential equations},
    journal={Ann. Scuola Norm. Sup. Pisa (3)},
    volume={13},
    date={1959},
    pages={115--162},
}

\bib{Ornstein_1962}{article}{
    author={Ornstein, Donald},
    title={A non-equality for differential operators in the $L_{1}$ norm},
    journal={Arch. Rational Mech. Anal.},
    volume={11},
    date={1962},
    pages={40--49},
    issn={0003-9527},
    doi={10.1007/BF00253928},
}

\bib{Raita_2019}{article}{
    author={Rai\c{t}\u{a}, Bogdan},
    title={Critical \(\mathrm{L}^p\)--differentiability of \(\mathrm{BV}^{\mathbb{A}}\)-maps and canceling operators},
    journal={Trans. Amer. Math. Soc.},
    volume={372},
    date={2019},
    number={10},
    pages={7297--7326},
    issn={0002-9947},
    doi={10.1090/tran/7878},
}

\bib{Raita_Skorobogatova_2020}{article}{
    author={Rai\c{t}\u{a}, Bogdan},
    author={Skorobogatova, Anna},
    title={Continuity and canceling operators of order $n$ on \(\mathbb{R}^n\)},
    journal={Calc. Var. Partial Differential Equations},
    volume={59},
    date={2020},
    number={2},
    pages={Paper No. 85, 17},
    issn={0944-2669},
    doi={10.1007/s00526-020-01739-z},
}

\bib{Roginskaya_Wojciechowski_2006}{article}{
   author={Roginskaya, Maria},
   author={Wojciechowski, Micha\l },
   title={Singularity of vector valued measures in terms of Fourier
   transform},
   journal={J. Fourier Anal. Appl.},
   volume={12},
   date={2006},
   number={2},
   pages={213--223},
   issn={1069-5869},
   doi={10.1007/s00041-005-5030-9},
}    


\bib{Schwartz_1950}{book}{
    author={Schwartz, L.},
    title={Th\'{e}orie des distributions},
      part={I},
      series={Publ. Inst. Math. Univ. Strasbourg},
      volume={9},
      publisher={Hermann \& Cie, Paris},
      date={1950},
      pages={148},
}

\bib{Schwartz_1951}{book}{
    author={Schwartz, L.},
    title={Th\'{e}orie des distributions},
    part={II},
    series={Publ. Inst. Math. Univ. Strasbourg},
    volume={10},
    publisher={Hermann \& Cie, Paris},
    date={1951},
   pages={169},
}

\bib{Smirnov_1993}{article}{
    author={Smirnov, S. K.},
    title={Decomposition of solenoidal vector charges into elementary solenoids, and the structure of normal one-dimensional flows},
    language={Russian},
    journal={Algebra i Analiz},
    volume={5},
    date={1993},
    number={4},
    pages={206--238},
    issn={0234-0852},
    translation={
        journal={St. Petersburg Math. J.},
        volume={5},
        date={1994},
        number={4},
        pages={841--867},
        issn={1061-0022},
    },
}

\bib{Sobolev_1938}{article}{
    author = {Sobolev, S.},
    title = {Sur un th\'eor\`eme d'analyse fonctionnelle},
    journal = {{Rec. Math. Moscou, n. Ser.}},
    volume = {4},
    pages = {471--497},
    Year = {1938},
    publisher = {Moscow Mathematical Society, Moscow},
    language = {Russian with French Summary},
}

\bib{Spector_VanSchaftingen_2019}{article}{
   author={Spector, Daniel},
   author={Van Schaftingen, Jean},
   title={Optimal embeddings into Lorentz spaces for some vector
   differential operators via Gagliardo's lemma},
   journal={Atti Accad. Naz. Lincei Rend. Lincei Mat. Appl.},
   volume={30},
   date={2019},
   number={3},
   pages={413--436},
   issn={1120-6330},
   doi={10.4171/RLM/854},
}



\bib{Stein_1970}{book}{
    author={Stein, Elias M.},
    title={Singular integrals and differentiability properties of functions},
    series={Princeton Mathematical Series, No. 30},
    publisher={Princeton University Press}, 
    address={Princeton, N.J.},
    date={1970},
    pages={xiv+290},
}

\bib{Stolyarov_2020}{article}{
    eprint={https://arxiv.org/abs/2010.05297},
    title={Hardy--Littlewood--Sobolev inequality for \(p=1\)},
    author={Stolyarov, Dimitriy},
}

\bib{Strauss_1973}{article}{
    author={Strauss, Monty J.},
    title={Variations of Korn's and Sobolev's equalities},
    conference={
      title={Partial differential equations},
      address={Univ. California,
        Berkeley, Calif.},
      date={1971},
    },
    book={
      series={Proc. Sympos. Pure Math.}, 
      volume={XXIII}, 
      publisher={Amer. Math. Soc.}, 
      address={Providence, R.I.},
    },
    date={1973},
    pages={207--214},
}

\bib{Talenti_1976}{article}{
   author={Talenti, Giorgio},
   title={Best constant in Sobolev inequality},
   journal={Ann. Mat. Pura Appl. (4)},
   volume={110},
   date={1976},
   pages={353--372},
   issn={0003-4622},
   doi={10.1007/BF02418013},
}

\bib{Timoshenko_Goodier_1951}{book}{
    author={Timoshenko, S.},
    author={Goodier, J. N.},
    title={Theory of Elasticity},
    edition={2},
    publisher={McGraw-Hill}, 
    address={New York--Toronto--London},
    date={1951},
    pages={xviii+506},
}    

\bib{VanSchaftingen_2004_circ}{article}{
    author={Van Schaftingen, Jean},
    title={A simple proof of an inequality of Bourgain, Brezis and Mironescu},
    journal={C. R. Math. Acad. Sci. Paris},
    volume={338},
    date={2004},
    number={1},
    pages={23--26},
    issn={1631-073X},
    doi={10.1016/j.crma.2003.10.036},
}
    
\bib{VanSchaftingen_2004_div}{article}{
    author={Van Schaftingen, Jean},
    title={Estimates for $L^1$-vector fields},
    journal={C. R. Math. Acad. Sci. Paris},
    volume={339},
    date={2004},
    number={3},
    pages={181--186},
    issn={1631-073X},
    doi={10.1016/j.crma.2004.05.013},
}

\bib{VanSchaftingen_2004_ARB}{article}{
    author={Van Schaftingen, Jean},
    title={Estimates for $L^1$ vector fields with a second order condition},
    journal={Acad. Roy. Belg. Bull. Cl. Sci. (6)},
    volume={15},
    date={2004},
    number={1--6},
    pages={103--112},
    issn={0001-4141},
}

\bib{VanSchaftingen_2006}{article}{
    author={Van Schaftingen, Jean},
    title={Function spaces between BMO and critical Sobolev spaces},
    journal={J. Funct. Anal.},
    volume={236},
    date={2006},
    number={2},
    pages={490--516},
    issn={0022-1236},
    doi={10.1016/j.jfa.2006.03.011},
}

\bib{VanSchaftingen_2008}{article}{
    author={Van Schaftingen, Jean},
    title={Estimates for $L^1$ vector fields under higher-order differential conditions},
    journal={J. Eur. Math. Soc. (JEMS)},
    volume={10},
    date={2008},
    number={4},
    pages={867--882},
    issn={1435-9855},
    doi={10.4171/JEMS/133},
}

\bib{VanSchaftingen_2013}{article}{
    author={Van Schaftingen, Jean},
    title={Limiting Sobolev inequalities for vector fields and canceling
      linear differential operators},
    journal={J. Eur. Math. Soc. (JEMS)},
    volume={15},
    date={2013},
    number={3},
    pages={877--921},
    issn={1435-9855},
    doi={10.4171/JEMS/380},
}

\bib{VanSchaftingen_2014}{article}{
    author={Van Schaftingen, Jean},
    title={Limiting Bourgain-Brezis estimates for systems of linear differential equations: theme and variations},
    journal={J. Fixed Point Theory Appl.},
    volume={15},
    date={2014},
    number={2},
    pages={273--297},
    issn={1661-7738},
    doi={10.1007/s11784-014-0177-0},
}

\bib{VanSchaftingen_2023}{article}{ 
    author={Van Schaftingen, Jean},
    title={Injective ellipticity, cancelling operators, and endpoint Gagliardo-Nirenberg-Sobolev inequalities for vector fields},
    eprint={http://arxiv.org/abs/2302.01201},
    note={Lecture notes for the CIME summer school “Geometric and analytic aspects of functional variational principles”, June 27 -- July 1, 2022},
}

\bib{VanSchaftingen_Yung_2022}{article}{
    author={Van Schaftingen, Jean},
    author={Yung, Po-Lam},
    title={Limiting Sobolev and Hardy inequalities on stratified homogeneous
    groups},
    journal={Ann. Fenn. Math.},
    volume={47},
    date={2022},
    number={2},
    pages={1065--1098},
    issn={2737-0690},
}

\bib{Willem_2013}{book}{
    author={Willem, Michel},
    title={Functional analysis},
    series={Cornerstones},
    subtitle={Fundamentals and applications},
    publisher={Birkh\"{a}user/Springer, New York},
    date={2013},
    pages={xiv+213},
    isbn={978-1-4614-7003-8},
    isbn={978-1-4614-7004-5},
    doi={10.1007/978-1-4614-7004-5},
}
  \end{biblist}

\end{bibsection}
\end{document}